\documentclass[12pt,reqno]{amsart}
\usepackage[a4paper,margin=1in]{geometry}

\usepackage{microtype}
\usepackage[cal=euler,scr=boondoxo]{mathalfa}
\usepackage{tikz,tikz-cd}
\usetikzlibrary{decorations}
\usepackage{enumitem}

\usepackage[colorlinks,allcolors=blue]{hyperref}
\usepackage[capitalize]{cleveref}
\crefname{equation}{}{}

\input{commands}

\def\bi{\mathbf i}
\def\bn{\mathbf n}

\oper{QM}
\oper{loc}
\oper{edeg}
\oper{Frac}

\oper{Pole}
\oper{CT}
\def\setc#1#2{\set{#1\col#2}}

\oper{vect}
\def\d{\mathbf d} 
\def\qd{q^\d}

\def\NOP{normally ordered product\xspace}

\def\pr#1#2{#1^{(#2)}}
\def\ph#1{(#1;q)_\infty} 

\def\hh{\what H} 
\opr\N{\mathbf N} 

\def\hb#1#2{#1^{[#2]}}
\def\fat{\fa^\ts}

\def\bil{\rbr}
\def\W{\sW}

\def\bch{\mathfrak{ch}}
\def\top{\mathrm{top}}
\def\Zs{\bZ_*}
\oper{sh}

\opr\hom{\mathbf{Hom}}
\opr\iend{\mathbf{End}}

\oper{str}

\def\gr{\mathrm{gr}}

\def\F{\cF} 
\def\FG{\cF^g} 
\def\Fo{\cF_0} 

\def\BF{\phi} 
\def\BD#1{\phi^{[#1]}} 
\def\UD{\mathbf A} 
\def\AD{A} 

\def\loc{\mathrm{loc}}
\def\Ho{\cH^\circ} 
\def\D{\fD}
\def\Db{\D'}

\def\b#1{\boldsymbol{#1}}
\oper{Par}
\def\tD{\wtl D}
\def\vsi{\varsigma}

\begin{document}
\title{Quasi-modularity of $q$-traces and integrals over Hilbert schemes}

\author{Killian Hong-Minh}
\author{Sergey Mozgovoy}

\address{School of Mathematics, Trinity College Dublin, Dublin 2, Ireland
\newline\indent
Hamilton Mathematics Institute, Dublin 2, Ireland}

\email{hongmink@tcd.ie}
\email{mozgovoy@maths.tcd.ie}

\begin{abstract}
We study quasi-modularity of normalized $q$-traces on bosonic Fock spaces associated with finite-dimensional quadratic spaces and superspaces.
For a natural class of operators obtained from free-boson fields and their descendants, we prove that their normalized $q$-traces are quasi-modular, with weight bounded by the sum of the weights of the insertions.
As applications, we prove Qin's quasi-modularity conjecture for tautological integrals on Hilbert schemes of points of a surface with numerically trivial canonical class, obtain quasi-modularity of arbitrary zero-mode correlation functions in the Heisenberg vertex operator algebra, and give a new proof of the Bloch--Okounkov quasi-modularity theorem as a rank-one specialization of our general result.
\end{abstract}

\maketitle

\section*{Introduction}

Modular and quasi-modular forms arise throughout enumerative geometry,
representation theory, and the theory of vertex operator algebras.
A common source of this phenomenon is genus one:
geometric generating series can often be represented as $q$-traces, interpreted as torus correlation functions, on a suitable Fock space.
For free bosonic theories these correlation functions are governed by
the Heisenberg algebra, and their basic two-point functions are
expressed in terms of theta functions and Eisenstein series.

The purpose of this paper is to formulate a general quasi-modularity
theorem for such $q$-traces.
Let $H$ be a finite-dimensional quadratic space, or more generally a
quadratic superspace, and let $V=V(H)$ be the corresponding bosonic
Fock space.
We consider a natural class of admissible fields
obtained from the free boson fields and their descendants
\[
\BF_a(z)=\sum_{n\ne0}\frac{a_{-n}}n z^n,
\qquad
\BD{k}_a(z)=D^k\BF_a(z),
\qquad
D=z\dd_z,
\]
and admissible operators obtained by taking their constant coefficients.
For admissible fields
$\cT_1(z_1),\dots,\cT_r(z_r)$ we study the normalized
$q$-correlation functions
\[
\ta\rbr{\cT_1(z_1)\dots\cT_r(z_r)}
=\frac{\tr_q\rbr{\cT_1(z_1)\dots\cT_r(z_r)}}{\tr_q(1)}.
\]

Our main result has two levels.
At the level of fields, these $q$-correlation functions belong to a
filtered algebra of quasi-elliptic functions in the insertion
variables.
At the level of operators, taking simultaneous constant coefficients
produces quasi-modular forms.
Schematically,
\[
\text{Heisenberg $q$-correlation functions}
\longrightarrow
\text{quasi-elliptic functions}
\overset{\CT}{\longrightarrow}
\text{quasi-modular forms}.
\]
We prove that if $T_1,\dots,T_r$ are admissible operators of weights
at most $w_1,\dots,w_r$, then their normalized $q$-trace
is quasi-modular of weight at most $w_1+\cdots+w_r$.
We prove the same statement for a modified normalized $q$-trace
with an additional weight-zero exponential Heisenberg insertion;
this version is used in the Hilbert-scheme application.

The proof reflects the Gaussian nature of the free bosonic theory.
Normalized $q$-traces make the space of linear Heisenberg fields
Gaussian, so all higher $q$-correlation functions are determined by
pairwise contractions.
For bosonic descendants these contractions are logarithmic derivatives
of the Jacobi theta function together with Eisenstein series.
This gives the quasi-elliptic field-level statement, while a filtered
constant term theorem preserves the weight bound and yields
quasi-modularity.

This point of view is closely related to genus-one correlation
functions in vertex operator algebras.
Zhu \cite{zhu_modular} established foundational modularity results
for genus-one trace functions,
while Mason--Tuite \cite{mason_torus}
obtained explicit genus-one $r$-point functions for free bosonic and
lattice vertex operator algebras.
Our emphasis is different: we organize Heisenberg $q$-correlation
functions into a filtered algebra of quasi-elliptic functions and study their constant coefficients,
obtaining quasi-modularity statements for $q$-traces of arbitrary products of admissible operators.

We give three applications.

\medskip
\clm1
\textit{Qin's conjecture for Hilbert schemes.}
Let $X$ be a smooth projective surface.
The Nakajima--Grojnowski construction identifies
\[
\bigoplus_{n\ge0}H^*(X^{[n]},\bQ)
\]
with the Fock representation of the Heisenberg algebra associated with
the quadratic superspace $H^*(X,\bQ)$.
Thus intersection theory on all Hilbert schemes $X^{[n]}$ is encoded
in a single, generally higher-rank, Heisenberg algebra representation.

For vector bundles $E_1,\dots,E_r$ on $X$, consider the generating
series
\[
\ang{\ch_{k_1}^{E_1}\dots\ch_{k_r}^{E_r}}
=\sum_{n\ge0}q^n\int_{\hb Xn}
\ch_{k_1}(\hb En_1)\dots\ch_{k_r}(\hb En_r)\, \fe_t(T_{\hb  Xn}),
\]
where $\fe_t(F)=\prod_i(t+x_i)$
is the reverse Chern polynomial
of a vector bundle $F$ with Chern roots $(x_i)_i$.
For example, for $r=0$,
we have
\[\ang{}=\sum_{n\ge0}\hi(\hb Xn)q^n=\ph q^{-\hi(X)},
\]
where $\ph z=\prod_{i\ge0}(1-zq^i)$ is the $q$-Pochhammer symbol and
$\hi(X)$
is the Euler characteristic of~$X$.
Define the reduced (or normalized) series
\[
\ang{\ch_{k_1}^{E_1}\dots\ch_{k_r}^{E_r}}'
={\ang{\ch_{k_1}^{E_1}\dots\ch_{k_r}^{E_r}}}/\ang{}.
\]
For line bundles $E_i$ and $t=1$,
it was conjectured by Okounkov \cite{okounkov_hilbert}
that this series is a multiple $q$-zeta value,
and it was conjectured by Qin \cite{qin_quick,alhwaimel_hilbert}
that this series is quasi-modular if $K_X$ is numerically trivial.
See \cite{carlsson_vertex,carlsson_vertexa,alhwaimel_qins,alhwaimel_equivariant} for related results.
For a numerically trivial $K_X$ and arbitrary vector bundles $E_i$,
our results imply that $\ang{\ch_{k_1}^{E_1}\dots\ch_{k_n}^{E_n}}'$
is quasi-modular of weight at most $\sum_i(k_i+2)$.

While preparing this manuscript, we became aware of the preprint
\cite{alekseev_qins}, which independently proves Qin's conjecture
using closely related Wick and constant-term methods, directly in the
Hilbert-scheme setting.
We were also informed by Dominic Joyce that he has an alternative method of proving quasi-modularity of integrals over Hilbert schemes.

\medskip
\clm2
\textit{Zero-mode correlation functions of the Heisenberg vertex
operator algebra.}
Let $V=V(H)$ be the Heisenberg vertex operator algebra associated with
an arbitrary finite-dimensional quadratic space $H$.
For a homogeneous state $v\in V_d$, let
\[
o(v)=v_{(d-1)}
=[z^0]Y(z^{L_0}v,z)
\]
be its zero mode.
We prove that the admissible operators considered in this paper are
precisely the image of the zero-mode map $o:V\to\End(V)$.
Moreover, the admissible fields are precisely the degree-zero fields
$Y(z^{L_0}v,z)$.

It follows that for arbitrary states $v_1,\dots,v_r\in V$,
the zero-mode $r$-point correlation function
\[
\ta\bigl(o(v_1)\dots o(v_r)\bigr)
\]
is quasi-modular.
For $r=1$ this was proved by
Dong--Mason--Nagatomo \cite{dong_quasi}; our theorem treats arbitrary
products of zero modes together with their natural weight filtration.

\medskip
\clm3
\textit{Bloch--Okounkov quasi-modularity results.}
Bloch and Okounkov \cite{bloch_character} introduced $r$-point
correlation functions on the charge-zero infinite wedge and proved
quasi-modularity of $q$-brackets of shifted symmetric functions.
Under the boson--fermion correspondence, the charge-zero infinite
wedge is the \emph{rank-one} bosonic Fock space,
which in turn is isomorphic to the algebra~$\La$ of symmetric functions, with the basis consisting of Schur functions $\ket\la=s_\la$ for partitions $\la\in\Par$.

The correlation functions of Bloch and Okounkov form a distinguished rank-one specialization of the Heisenberg $q$-correlation functions considered in this paper.
For a function $f:\Par\to\bQ$ one may consider the operator $\hat f$ on $\La$ given by
\[
\hat f\ket\la=f(\la)\ket\la.
\]
Then the $q$-bracket of $f$ is exactly the normalized $q$-trace:
\[
\ang f_q
=\frac{\sum_\la f(\la)q^{\n\la}}{\sum_\la q^{\n\la}}
=\ta(\hat f).
\]
In \cite{bloch_character} one introduces functions
$Q_k:\Par\to\bQ$, which generate the algebra
$\La^*=\bQ[Q_2,Q_3,\dots]$ of shifted symmetric functions,
with $\wt(Q_k)=k$.
One proves that the $q$-brackets $\ang{Q_{k_1}\dots Q_{k_r}}_q$
are quasi-modular.
We will show that the corresponding operator $\what Q_k$ on the Fock
space $\La$ is admissible of weight at most $k$.
Therefore
$\ta(\what Q_{k_1}\dots\what Q_{k_r})=\ang{Q_{k_1}\dots Q_{k_r}}_q$
is quasi-modular of weight at most $\sum_i k_i$.
This gives a new proof of the quasi-modularity, together with the standard upper weight bound, of Bloch–Okounkov $q$-brackets.

\medskip

The paper is organized as follows.
We first develop the Heisenberg and Fock-space formalism and the
corresponding normalized $q$-traces and $q$-correlation functions.
We then recall the Nakajima description of Hilbert-scheme cohomology
and the operators of Li--Qin--Wang and Carlsson--Okounkov.
Next we prove Gaussianity of the Heisenberg $q$-correlation functions,
construct the filtered quasi-elliptic algebras, and establish the
constant term and quasi-modularity theorems.
We conclude with the three applications above: Qin's conjecture,
zero-mode correlation functions, and the Bloch--Okounkov theorem.

\subsection*{AI Declaration}
An LLM was used for exploratory discussions and editorial assistance during the preparation of this paper.

\subsection*{Acknowledgments}
The authors would like to thank
Dominic Joyce and Jan Manschot for many useful discussions.

\section{Fock spaces, fields and \tpdf{$q$}{q}-traces}

\subsection{Heisenberg algebra and Fock space}
Let $K$ be a field of characteristic zero.
A quadratic space over $K$ is a \fdd vector space $H$
equipped with a non-degenerate symmetric bilinear form $(-,-)$.
The \idef{Heisenberg algebra} of $H$ is the graded Lie algebra
\begin{equation}
\hh
=H[t^{\pm1}]\oplus K\bc
\end{equation}
with the central element $\bc$ and the Lie bracket
\begin{equation}\label{H-bracket}
[a_m,b_n]=m\de_{m,-n}(a,b)\bc,\qquad
a_m=at^m,\,b_n=bt^n,
\end{equation}
where $a,b\in H$ and $m,n\in\bZ$.
The degree on $\hh$ (also called conformal weight or energy)
is defined by
\[\deg a_n=-n,\qquad \deg\bc=0.\]

Let $\hh_+=H[t]$ and $\hh_-=t\i H[t\i]$.
Define the (bosonic) \idef{Fock space} to be the $\hh$-module
\begin{equation}
V=V(H)=\Ind_{\hh_+\oplus K\bc}^{\hh}K\iso S(\hh_-),
\end{equation}
where $\hh_+$ acts as zero
and $\bc$ acts as identity on $K$.
Since the pairing on $H$ is non-degenerate,
the module $V$ is irreducible over $\hh$.
The element $\one=\ket0\in V$ corresponding to $1\in K$ is called
the \idef{vacuum}.
The grading on $\hh$ induces the grading
\[V=\bop_{n\ge0}V_n,\]
where all $V_n$ are \fdd and $V_0=K\one$.
The operators $a_n$ for $n\ge0$ annihilate $\one$, and are called annihilation operators.
The operators $a_n$ for $n<0$ are called creation operators.

\begin{remark}
Let $P(V,q)=\sum_{n\ge0}\dim(V_n)q^n$
be the Poincar\'e series of $V$.
If $d=\dim H$,
then $P(Ht^{-n},q)=dq^n$, hence $P(S(Ht^{-n}),q)=\frac1{(1-q^n)^d}$.
Therefore
\begin{equation}
P(V,q)
=\prod_{n\ge1}\frac1{(1-q^n)^d}
=\frac1{\ph q^d}.
\end{equation}
\end{remark}

\begin{lemma}\label{rho}
The map
\[\rho_n:Ht^n\to\End(V),\qquad at^n\mto a_n,\]
is injective for all $n\ne0$.
Moreover, $\rho_0=0$.
\end{lemma}
\begin{proof}
If $n<0$ and $a\ne 0$, then $a_n\one\ne0$.
Assume that $n>0$ and $a_n=0$.
For all $b\in H$, we have
$0=a_nb_{-n}\one=[a_n,b_{-n}]\one=n(a,b)\one$, hence $(a,b)=0$.
Therefore $a=0$.
\end{proof}

\subsection{Graded vector spaces}
Given a category $\cC$ and a set $I$,
let $\cC^I=[I,\cC]$ denote the category of $I$-graded objects in~$\cC$,
with the objects being collections
$X=(X_i\in\cC)_{i\in I}$ (also written as $\bop_{i\in I}X_i$)
and morphisms
\[\Hom_{\cC^I}\rbr{X,Y}
=\prod_{i\in I}\Hom_\cC(X_i,Y_i).\]
If $\cC$ is an additive category (or has an initial object $0$),
define the restricted product $\cC^{(I)}$ to be the full subcategory of $\cC^I$ with objects $\bop_{i\in I}X_i$ such that $X_i\iso0$ for all but finitely many $i\in I$.

If $\Ga$ is a (cancellative) commutative monoid, $X\in\cC^\Ga$ and $n\in \Ga$,
define the shift
\[X[n]\in\cC^\Ga,\qquad X[n]_i=X_{n+i},\, i\in \Ga.\]

In particular, let $\Vect_K$ be the category of vector spaces over a field $K$ and let $\vect_K$ be the
category of \fdd vector spaces over $K$.
Then the category $\Vect_K^\Ga$ is a closed symmetric monoidal category,
with tensor products, internal Homs and symmetry defined by
\begin{gather}
V\ts W=\bop_{n\in\Ga}(V\ts W)_n,\qquad
(V\ts W)_n=\bop_{i+j=n}V_i\ts W_{j},\\
\hom(V,W)=\bop_{n\in\Ga}\hom_n(V,W),\qquad
\hom_n(V,W)
=\prod_{i\in\Ga}\Hom_K(V_i,W_{n+i}),\\
\si:V\ts W\to W\ts V,\qquad a\ts b\mto b\ts a.
\end{gather}
Define the $\Ga$-graded algebra $\iend(V)=\hom(V,V)$.
The subcategory $\vect_K^{(\Ga)}\sbs\Vect_K^\Ga$ is also closed monoidal.
For $V,W\in\Vect_K^\Ga$, we have
\[\Hom_{\Vect_K^\Ga}(V,W)=\hom_0(V,W)=\prod*_{i\in\Ga}\Hom_K(V_i,W_i)\]
which is different from the Hom space $\Hom_K(V,W)\iso\prod_{i\in\Ga}\Hom_K(V_i,W)$
in the category $\Vect_K$.
To avoid confusion, we will denote the former by $\hom_0(V,W)$ and the latter by $\Hom_K(V,W)$.
A $\bZ$-graded or $\bN$-graded vector space will be called graded.

\subsection{Superspaces}
For more details on superspaces see \eg
\cite[p.~41]{deligne_quantum}.
Consider the group $C_2=\bZ/2\bZ$.
The category $\Vect_K^{C_2}$ (resp.\ $\vect_K^{C_2}$) is called the category of superspaces (resp.\ \fdd superspaces).
Given a superspace $V=\bop_{i\in C_2}V^i$ and $a\in V^i$,
we call $p(a)=i$ the parity (or $C_2$-degree) of $a$.
Whenever we write $p(a)$, we automatically assume that~$a$ is homogeneous.
The elements of $V^{\b0}$ are called even
and the elements of $V^{\b1}$ are called odd.
The category $\Vect_K^{C_2}$ is equipped with symmetry
(the Koszul sign rule)
\[\si:V\ts W\to W\ts V,\qquad a\ts b\mto(-1)^{p(a)p(b)}b\ts a.\]
In what follows,
we write $\ep(a,b)=(-1)^{p(a)p(b)}$ and automatically assume that $a,b$ are homogeneous.
An (associative, unital) superalgebra $A$ is a monoid in the category $\Vect_K^{C_2}$.
It is called (super) commutative if the product $\mu:A\ts A\to A$ satisfies $\mu\si=\mu$.
If $A$ is a superalgebra, then it has a (super) Lie bracket
\eq{\label{super Lie}
[a,b]=(\mu-\mu\si)(a\ts b)
=ab-(-1)^{p(a)p(b)}ba.}

A pairing $\be=(-,-):V\ts V\to K$ is called super-symmetric if $\be\si=\be$, meaning that
\[(a,b)=(-1)^{p(a)p(b)}(b,a).\]
It is called even if it has degree zero, meaning that
\[
(a,b)\ne0\imp p(a)=p(b).
\]
A pairing is even and super-symmetric \ifft
\[(a,b)=(-1)^{p(a)}(b,a).\]

Define a quadratic superspace over $K$ to be a superspace $H\in\vect_K^{C_2}$ equipped with a non-degenerate,
even and super-symmetric bilinear form $(-,-)$.
The Heisenberg algebra $\hh=H[t^{\pm1}]\oplus K\bc$
is a (super) Lie algebra
equipped with the bracket~\cref{H-bracket}.
The Fock space
\[V=V(H)=\bop_{n\in\bN}\bop_{i\in C_2}V_n^i
\iso S(\hh_-)\iso S(\hh_-^{\b0})\ts \La(\hh_-^{\b1})
\]
is both $\bN$-graded and $C_2$-graded (an $\bN$-graded superspace).
It is a module over $\hh$ so that
\[[a_m,b_n]=a_mb_n-(-1)^{p(a)p(b)}b_na_m=m\de_{m,-n}(a,b)\id_V\]
in $\iend(V)\in\Vect_K^{\bZ\xx C_2}$.
The (super) Poincar\'e series of $V$ is
\[P(V,q)=\sum_{n\ge0}(\dim V_n^{\b0}-\dim V_n^{\b1})q^n
=\ph q^{-\dim H^{\b0}+\dim H^{\b1}}.
\]

\subsection{Graded fields}
Given a vector space $V$,
let
\eq{V\lser z
=\sets{\sum*_{n\in\bZ}v_nz^n}{v_n\in V,\, v_n=0\text{ for }n\ll0}
\iso \bigcup_N\prod_{i\ge N}Vz^i.
}
Define the \idef{space of fields} on $V$ to be
\eq{\cF(V)=\Hom_K(V,V\lser z).}
If $V=\bop_{i\in\bZ}V_i$ is a bounded below graded vector space,
we define $\deg z=-1$ and
\eq{V^\gr\lser z
=\bop_{d\in\bZ}\rbigg{\prod_{i\in\bZ}V_{d+i}z^{i}}\sbs V\lser z.
}
We define the \idef{space of graded fields} on $V$
\begin{gather}
\FG(V)
=\hom(V,V^\gr\lser z)
=\bop_{d\in\bZ}\F_d(V),\\
\F_d(V)
=\sets{\sum*_{n\in\bZ} f_nz^n:V\to V\lser z}
{\deg f_n-n=d\ \forall n}.
\end{gather}
In particular, let
\[\Fo(V)
=\sets{\sum*_{n\in\bZ} f_nz^n:V\to V\lser z}
{\deg f_n=n\ \forall n}\iso\prod*_{n\in\bZ}\iend_n(V)
\]
be the space of degree zero fields.

\begin{example}
Let $V=V(H)$ be the Fock space over the Heisenberg algebra $\hh$.
For every $a\in H$, the Heisenberg field
\eq{
Y(a,z)=\sum_{n\in\bZ}a_n z^{-n-1}
}
has degree $1$, since $\deg(a_nz^{-n-1})=-n+(n+1)=1$.
The \idef{free boson field} (without the zero mode)
\eq{\label{free bos}
\BF_a(z)=\sum_{n\ne0}\frac{a_{-n}}{n}z^n}
has degree zero.
More generally, the fields
\eq{\label{descend}
\BD k_a(z)=D^k\BF_a(z)
=\sum_{n\ne0}a_{-n}n^{k-1}z^n,\qquad
k\ge0,\, D=z\dd_z,
}
have degree zero.
We call them bosonic descendants of $\BF_a(z)$.
\end{example}

\subsection{Polarized algebras and normally ordered products}
In this section we develop a purely algebraic framework for \NOP{s}. The basic issue addressed by normal ordering is that arbitrary fields in $\cF(V)$ need not be composable.
A \NOP provides a canonical way to reorder the factors so that their composition is well defined.
We introduce \NOP{s} in a more general setting than $\F(V)$, isolating only the algebraic structure needed for their definition.

First, let $A=A_+\oplus A_-$ be an associative algebra, where $A_{\pm}$ are (not necessarily unital) subalgebras.
Every $a\in A$ has a unique decomposition $a=a_++a_-$
with $a_\pm\in A_\pm$ (sometimes we write $a^\pm$ instead).
Given $a_1,\dots,a_n\in A$,
define their \idef{\NOP}
\begin{equation}\label{nop}
\N(a_1\ts\dots\ts a_n)
=\nop{a_1\dots a_n}
:=\sum_{I\sqcup J=[n]}a^+_{I}a^{-}_{J},
\end{equation}
where
$a^\pm_{I}=a^\pm_{i_1}\dots a^\pm_{i_k}\in A$ for $I=\set{i_1<\dots<i_k}\sbs [n]$.

\begin{remark}
Note that traditionally the \NOP is defined by
\[\lb{\nop{ab}}=a_+b+ba_-\]
and then extended recursively.
See \eg \cite[2.3.5]{kac_vertex}, \cite[2.2.2]{frenkel_vertex},
\cite[3.8.5]{lepowsky_introduction}.
We have $\nop{ab}-\lb{\nop{ab}}=[a_-,b_-]$.
This commutator vanishes for Heisenberg fields.
\end{remark}

\begin{definition}
A \idef{polarized algebra} is a pair $(A_+,A_-)$ of associative (not necessarily unital) algebras together with a linear map
$\mu:A_+\ts A_-\to A=A_+\oplus A_-$
such that
\begin{enumerate}
\item the projection $\mu_+:A_+\ts A_-\to A_+$ is a morphism of left $A_+$-modules.
\item the projection $\mu_-:A_+\ts A_-\to A_-$ is a morphism of right $A_-$-modules.
\end{enumerate}
We will usually refer to $A=A_+\oplus A_-$ as a polarized algebra
and denote $\mu(a, b)$ by $ab$.
We assume that $A_+$ is unital and $\mu(1,a)=a$ for all $a\in A_-$.
Then the \NOP~\cref{nop} induces a linear map
\[\N:\bop_{n\ge0}A^{\ts n}\to A.\]
\end{definition}

\begin{remark}
A polarized superalgebra is defined in the same way.
The \NOP in $A=A_+\oplus A_-$ is defined by
\[
\N(a_1\ts\dots\ts a_n)
=\nop{a_1\dots a_n}
:=\sum_{I\sqcup J=[n]}\ep(I,J,a) a^+_{I}a^{-}_{J},
\]
where $\ep(I,J,a)=\prod_{i\in I,j\in J,i>j}\ep(a_i,a_j)$.
\end{remark}

For (even) $a\in A$, define the \idef{normally ordered exponential}
(assuming that $K$ has characteristic zero)
\begin{equation}
\nop{e^{ta}}
=e^{ta_+}e^{ta_-}
=\sum_{n\ge0}\frac{t^n}{n!}
\nop{\underbrace{a\dots a}_n}\in A\pser t.
\end{equation}

\begin{example}\label{gr alg pol}
Given a graded algebra $A=\bop_{n\in\bZ}A_n$,
define its \idef{degree polarization} by
\begin{equation}\label{super-polar}
A_+=\bop_{n\ge0}A_n,\qquad A_-=\bop_{n<0}A_n.
\end{equation}
Every $a\in A$ can be uniquely written in the form $a=\sum_{n\in\bZ}a_n$, where $a_n\in A_n$.
Then $a_+=\sum_{n\ge0}a_n$ and $a_-=\sum_{n<0}a_n$.
\end{example}

\begin{example}
Let $\F(V)$ be the space of fields on a vector space $V$.
It is equipped with the \idef{field polarization} $\F(V)=\F(V)_+\oplus \F(V)_-$ given by
\begin{equation}\label{field-polar}
\F(V)_+=\Hom(V,V\pser z),\qquad
\F(V)_-=\Hom(V,z\i V[z\i]).
\end{equation}
Both $\F(V)_\pm$ are algebras and there is a well-defined
composition product
\[\F(V)_+\ts \F(V)_-\to \F(V).\]
In contrast, the composition $\F(V)_-\ts \F(V)_+\to \F(V)$ is not well-defined
in general.
\end{example}

\begin{example}
Let $V=\bop_{i\in \bZ}V_i$ be a bounded below graded vector space.
Then the spaces of graded fields $\FG(V)$ and degree zero fields $\Fo(V)$ inherit polarization from $\F(V)$.
Thus
\begin{gather*}
\Fo(V)_+
=\sets{\sum_{n\ge0}f_nz^n:V\to V\lser z}{\deg f_n=n\ \forall n}
\iso\prod_{n\ge0}\iend_n(V),\\
\Fo(V)_-
=\sets{\sum_{n<0}f_nz^n:V\to V\lser z}{\deg f_n=n\ \forall n}
\iso\prod_{n<0}\iend_n(V).
\end{gather*}
Note that this polarization is compatible
with the polarization of the graded algebra $\iend(V)\sbs\prod_{n\in\bZ}\iend_n(V)\iso\Fo(V)$
from \cref{gr alg pol}.
\end{example}

\begin{lemma}\label{exp field}
Let $V$ be a bounded below graded vector space
and let $f=\sum_{n\ne0} f_nz^n\in\Fo(V)$.
Then
\[\nop{e^f}=e^{f_+}e^{f_-}\]
is a well-defined field of degree zero.
\end{lemma}
\begin{proof}
We have $f_+=\sum_{n>0}f_nz^n$ and $f_-=\sum_{n<0}f_nz^n$,
where $\deg f_n=n$ for all $n$.
If $v\in V$, then $f_-^kv=0$ for $k\gg0$, since $V$ is bounded below.
Therefore $e^{f_-}v\in V[z\i]$ is well-defined.

On the other hand,
$e^{f_+}:V\to V\pser z$ is well-defined.
It may be applied coefficient-wise to
$e^{f_-}v\in V[z\i]$, giving
$e^{f_+}e^{f_-}v\in V\lser z$.
Finally, every term $f_nz^n$ has degree zero, so
$e^{f_+}e^{f_-}$ has degree zero.
\end{proof}

\def\Co{\cC'}
\subsection{Graded trace}
Given a category $\cC$, let $\Co$ 
be the category of pairs $(V,f)$, where $V\in\cC$ and $f\in\End_\cC(V)$,
with morphisms $u:(V,f)\to(W,g)$ given by
$u\in\Hom_\cC(V,W)$ satisfying $uf=gu$.
If $\cC$ is symmetric monoidal, then so is $\Co$.

In particular, let $\Ga$ be a (cancellative) commutative monoid and let $\cC=\vect_K^{(\Ga)}$ be the category of \fdd $\Ga$-graded vector spaces.
Then $\Co$ is a symmetric monoidal category and the Grothendieck group $K_0(\Co)$ inherits the commutative ring structure.
The graded trace
\eq{\tr_\Ga:K_0(\Co)\to K[\Ga],\qquad (V,f)\mto\sum_{\ga\in\Ga}\tr(f|_{V_\ga})e^\ga,}
is a ring morphism.
For every $K$-algebra $L$ with multiplicative character
$\hi:\Ga\to L^\xx$, there is an algebra morphism $K[\Ga]\to L$, $e^\ga\mto\hi(\ga)$.
Combining it with $\tr_\Ga$, we obtain the ring morphism
\eq{\tr_\hi:K_0(\Co)\to L,\qquad (V,f)\mto\tr_\hi(f)=\sum_{\ga\in\Ga}\hi(\ga)\tr(f|_{V_\ga}).}

In particular, let $\Ga=C_2=\bZ/2\bZ$,
and consider the sign character
$\sgn:C_2\to\set{\pm1}$, $i\mto(-1)^i$.
For $\cC=\vect_K^{C_2}$, the ring morphism
\eq{\str:K_0(\Co)\to K,\qquad (V,f)\mto \str(f)=\sum_{i\in C_2}(-1)^i\tr(f|_{V^i}),}
is called the supertrace.

Let $\Ga=\bZ$, $L=K[q^{\pm1}]$,
and consider the character $\hi:\bZ\to L^\xx$, $n\mto q^n$.
For $\cC=\vect_K^{(\bZ)}$, the ring morphism
\eq{\tr_q:K_0(\Co)\to L,\qquad
(V,f)\mto
\tr_q(f)=
\sum_{n\in\bZ}\tr(f|_{V_n})q^n,
}
is called the $q$-trace.

Let $\Ga=\bZ\xx C_2$, $L=K[q^{\pm1}]$,
and consider the character $\hi:\Ga\to L^\xx$, $(n,i)\mto(-1)^i q^n$.
The objects $V=\bop_{n\in\bZ}V_n=\bop_{n\in\bZ}\bop_{i\in C_2}V_n^i$
in $\cC=\vect_K^{(\bZ\xx C_2)}$ are called graded superspaces.
The ring morphism
\eq{\str_q:K_0(\Co)\to L,\qquad
(V,f)\mto\str_q(f)=\sum_{n\in\bZ}\str(f|_{V_n})q^n,
}
is called the (super) $q$-trace.

In the next sections we will discuss supertraces and $q$-traces in more detail.

\subsection{Supertrace properties}
If $Y$ is a smooth projective variety over $\bC$,
then the vector space $V=H^*(Y,\bQ)$ is cohomologically graded
and inherits a $C_2$-grading so that
$\hi(Y)=\sum_i(-1)^{i}\dim(V_i)=\str(\id_V)$,
where $\str$ is the supertrace.
More generally, if $V=\bop_{i\in\bZ}V^i$ is a \fdd $\bZ$-graded vector space,
then $V$ is also $C_2$-graded, where for $a\in V^i$ we define its parity ($C_2$-degree) $p(a)=i\pmod 2$.

Let $V=\bop_{i\in C_2}V^i\in\vect_K^{C_2}$ be a superspace.
For $f\in\End(V)$,
its supertrace is
\eq{
\str(f)=\tr(f_{00})-\tr(f_{11}),}
where $f_{ij}\in\Hom(V^i,V^j)$ are components of $f$
for $i,j\in C_2$.

If $V$ is a quadratic superspace,
we equip $V\ts V$ with the pairing
\eq{(a\ts b,x\ts y)=(-1)^{p(b)p(x)}(a,x)(b,y).}
For $f\in\End(V)$,
define its reproducing kernel
to be an element $C_f\in V\ts V$ such that
\eq{\label{repr ker}
(fa,b)=(a\ts b,C_f)\qquad \forall a,b\in V.}
If $(e_i)_i$ and $(e^i)_i$ are dual bases of $V$ satisfying $(e_i,e^j)=\de_{ij}$
and $f$ is even, then
\eq{C_{f}=\sum_i e^i\ts fe_i,}
since $(e_i\ts e^j,C_{f})=(-1)^{p(e_i)p(e_j)}(e^j,fe_i)=(fe_i,e^j)$.

\begin{lemma}\label{str and C}
Let $V$ be a quadratic superspace, $f\in\End(V)$ and $C_f=\sum_k x_k\ts y_k$ be the reproducing kernel of $f$.
Then
\[\str(f)=\sum_k(x_k,y_k).\]
\end{lemma}
\begin{proof}
We can assume that $f$ is even
(if $f$ is odd, then $C_f$ is also odd and $\str(f)=0=\sum_k(x_k,y_k)$).
Therefore $C_f=\sum_k x_k\ts y_k=\sum_ie^i\ts fe_i$, hence
\[\str(f)=\sum_i (-1)^{p(e_i)}(fe_i,e^i)
=\sum_i (e^i,fe_i)=\sum_k(x_k,y_k).
\qedhere
\]
\end{proof}

\subsection{\tpdf{$q$}{q}-traces}
Let $K$ be a field of characteristic zero,
and let $V=\bop_{n\in\bZ}V_n\in\vect_K^\bZ$ be a nonzero bounded below $\bZ$-graded vector space over $K$ with $\dim V_n<\infty$ for all $n$.
Define the degree operator $\d$ on $V$ by $\d|_{V_n}=n\cdot\id$,
and define the operator
\[
q^\d:V\to V\ts K[q^{\pm1}],\qquad V_n\ni x\mto q^n x.
\]
For $f\in\iend_n(V)$, we have
\begin{equation}\label{qd-comm}
q^\bd f=q^n\cdot fq^\bd.
\end{equation}

Given $f\in\iend(V)$, define the $q$-trace
\eq{\tr_q(f)
=\tr(q^\d f)=\sum_{n\in\bZ}\tr(f_{nn})q^n
\in K\lser q,}
where $f_{ij}\in\Hom(V_i,V_j)$ is the component of $f$ for $i,j\in\bZ$.
Note that $\tr_q(f)$ depends only on the degree zero component $f_0\in\iend_0(V)$ of $f$.
We obtain a linear map
\[\tr_q:A=\iend(V)\to K\lser q\]
such that
\[\tr_q(1)=\sum_{n\in\bZ}\dim(V_n)q^n=P(V,q)\]
is the Poincar\'e series of $V$.
Define the normalized $q$-trace
\eq{\label{norm q-trace}
\ta:A=\iend(V)\to K\lser q,\qquad f\mto\frac{\tr_q(f)}{\tr_q(1)}
=\frac{\tr(q^\bd f)}{\tr(q^\bd)},
}
so that $\ta(1)=1$.

\begin{lemma}
\label{ta=0}
Let $a\in A_n$ and $b\in A$.
Then
\begin{enumerate}
\item $\ta(ab)=q^n\cdot\ta(ba)$.
\item
If $n\ne0$ and $[a,b]=0$, then
$\ta(ab)=\ta(a)=0$.
\end{enumerate}
\end{lemma}
\begin{proof}
\clm1
We can assume that $b\in A_{-n}$, hence $b\qd=q^n\cdot \qd b$.
Therefore
\[\ta(ab)
=\frac{\tr(\qd ab)}{\tr(\qd)}
=\frac{\tr(b\qd a)}{\tr(\qd)}
=q^n\frac{\tr(\qd ba)}{\tr(\qd)}
=q^n\ta(ba).
\]

\clm2
We have $\ta(ab)=q^n\ta(ba)=q^n\ta(ab)$, hence $\ta(ab)=0$.
For $b=1$, we obtain $\ta(a)=0$.
\end{proof}

\begin{lemma}\label{ta-exp}
Let $a\in A_k$ and $b\in A_{-k}$ for $k\ne0$ be such that $c=[b,a]\in K$.
Let $d\in A$ commute with $a,b$.
Then
\[\ta(e^{sa} e^{tb} d)
=\exp\rbr{\frac{q^kstc}{1-q^{k}}}\ta(d)\in K\lser q\pser{s,t}.\]
\end{lemma}
\begin{proof}
If $m\ne n$, then $\deg(a^mb^n)=k(m-n)\ne0$
and $a^mb^n$ commutes with $d$,
hence $\ta(a^mb^nd)=0$ by \cref{ta=0}.
We have $b^na=ab^n+ncb^{n-1}$, since $\ad_a(b^n)=-ncb^{n-1}$.
Therefore $x_n=\ta(a^nb^nd)$ satisfies
\[x_n
=q^k\ta(a^{n-1}b^nda)
=q^k\rbr{\ta(a^nb^nd)+nc \ta(a^{n-1}b^{n-1}d)}
=q^k\rbr{x_n+ncx_{n-1}}.
\]
This implies $x_n=n\frac{q^kc}{1-q^k}x_{n-1}$.
Since $x_0=\ta(d)$, we obtain $x_n=n!\rbr{\frac{q^kc}{1-q^{k}}}^n\ta(d)$, hence
\[\ta(e^{sa}e^{tb}d)
=\sum_{n\ge0}\frac{s^nt^nx_n}{(n!)^2}
=\exp\rbr{\frac{q^kstc}{1-q^k}}\ta(d).
\qedhere
\]
\end{proof}

As an application,
let $V=V(H)$ be the Fock space for a quadratic vector space $H$.
For $a\in H$, consider the free boson field $\BF_a(z)$
\cref{free bos} with creation and annihilation parts
\[
\BF_a(z)_+=\sum_{n>0}\frac{a_{-n}}{n}z^n,\qquad
\BF_a(z)_-=\sum_{n>0}\frac{a_{n}}{-n}z^{-n}.
\]

\begin{lemma}\label{ta(f)}
Let $a,b\in H$ and let $f=e^{\BF_a(z)_+}e^{\BF_b(w)_-}$.
Then
\[\ta(f)=\ph{qz/w}^{(a,b)}.\]
\end{lemma}
\begin{proof}
We have $V\iso S(t\i H[t\i])\iso\bts_{n\ge1}S(Ht^{-n})$.
For $n\ge1$, let $a'=\frac{a_{-n}}{n}z^n$
and $b'=\frac{b_n}{-n}w^{-n}$.
Then $c=[b',a']=-\frac{(a,b)}n(z/w)^n$.
Applying \cref{ta-exp} to the product over $1\le n\le N$ and then
passing to the limit $N\to\infty$ coefficientwise, we obtain
\[\ta(f)
=\prod_{n\ge1}\exp\rbr{\frac{-(a,b)(qz/w)^n}{n(1-q^n)}}
=\ph{qz/w}^{(a,b)}.
\qedhere
\]
\end{proof}

\begin{remark}
Similarly, let $V=\bop_{n\in\bZ}V_n
=\bop_{n\in\bZ}(V_n^0\oplus V_n^1)\in(\vect_K^{C_2})^{\bZ}$
be a bounded below graded superspace.
Then we define the (super) $q$-trace
\[\str_q:A=\iend(V)\to K\lser q,\qquad
f\mto\str(q^\d f)
=\sum_{n\in\bZ}\str(f_{nn})q^n,
\]
where $f_{ij}\in\Hom(V_i,V_j)$ is the component of $f$ for $i,j\in\bZ$.
Assuming that $\str_q(1)\ne0$, define the normalization
\eq{\label{norm super-q-trace}
\ta:A=\iend(V)\to K\lser q,\qquad f\mto \frac{\str_q(f)}{\str_q(1)}.}
The preceding $q$-trace identities have corresponding super versions
(where exponential identities are applied to even elements).
All of the above results admit superspace analogues.
In particular, if $a\in A_n$ and $b\in A$ are parity-homogeneous, then
\[\ta(ab)=\ep(a,b)q^n\cdot\ta(ba).\]
The analogue of \cref{ta-exp} can be formulated for even $a,b$.
\end{remark}

\section{Integrals over Hilbert schemes}

\subsection{Nakajima operators}
Let $X$ be a smooth projective surface over $\bC$.
Then $H=H^*(X,\bQ)$ is a quadratic superspace
equipped with the pairing
\begin{equation}\label{pair}
\rbr{a,b}
=-\int_X ab,\qquad a,b\in H.
\end{equation}

Let $\hb Xn=\Hilb^n(X)$ be the $n$-th Hilbert scheme of $X$, and let
\begin{equation}
\bH=\bop_{n\ge0}\bH_n,\qquad
\bH_n=H^*(\hb Xn,\bQ).
\end{equation}
The identity element $\one\in\bH_0=H^*(\hb X0,\bQ)\iso\bQ$
is called the \idef{vacuum}.
For $n>0$ and $a\in H$, define the operator
$\fa_{-n}(a):\bH\to\bH$ of degree $n$,
called the \idef{Nakajima operator}, as follows.
For $m\ge0$, let
\[\hb Q{m+n,m}
=\sets{(\ksi,x,\eta)\in\hb X{m+n}\xx X\xx \hb Xm}
{\ksi\sps\eta,\,\supp(\ksi-\eta)=\set x}\]
and let $p_1,p_2,p_3$ denote the corresponding projections.
Then
\[\fa_{-n}(a):\bH_m\to\bH_{m+n},\qquad
u\mto p_{1*}\rbr{[\hb Q{m+n,m}]\cdot p_2^*a\cdot p_3^*u}.
\]
Let $\fa_{n}(a)=(-1)^n\fa_{-n}(a)^*$,
where $f^*:\bH_{m+n}\to \bH_m$ for $f:\bH_m\to\bH_{m+n}$
is the (super)adjoint \wrt the intersection pairings on $\bH_k$
\[(fu,v)=(-1)^{p(f)p(u)}(u,f^*v),\qquad u\in \bH_m,\, v\in\bH_{m+n}.\]
Let $\fa_0(a)=0$.
This defines linear maps
\[\fa_n:H\to\End(\bH),\qquad n\in\bZ.\]
In what follows,
we consider $\bH=\bop_{n\ge0}\bop_{i\in C_2}\bH_n^i$ as an $\bN$-graded superspace.
The (super) Lie bracket on $\iend(\bH)$ is defined as in
\cref{super Lie}.

\begin{theorem}
The Nakajima operators satisfy the commutation relations
\[[\fa_m(a),\fa_n(b)]=m\de_{m,-n}(a,b)\cdot \id_\bH,\]
hence define an action of the Heisenberg algebra $\hh$ on $\bH$.
Moreover, for the Fock space~$V$,
there is a unique $\hh$-module isomorphism
\[V\isoto \bH\]
sending the vacuum vector of $V$ to the vacuum vector of $\bH$.
In this way we identify the Nakajima operators $\fa_n(a)$
and the Heisenberg operators $a_n$ on the Fock space.
\end{theorem}

For every smooth projective variety $Y$,
the cohomology $H^*(Y,\bQ)$ is a (super) Frobenius algebra.
The idea which goes back to \cite{lehn_cup} is that the intersection theory of the Hilbert schemes~$\hb Xn$ is mostly controlled by the Frobenius algebra $H=H^*(X,\bQ)$.
Because of our choice of pairing~\cref{pair},
we equip $H=H^*(X,\bQ)$
with the coproduct and the counit
\begin{equation}
\de=-\De_*:H\to H\ts H,\qquad \eps=-\int_X:H\to\bQ,
\end{equation}
where $\De:X\to X\xx X$ is the diagonal.
We define $\de_n:H\to H^{\ts n}$
inductively by
$\de_0=\eps$, $\de_1=\id$ and $\de_{n+1}=(\de_n\ts\id)\de$.
If $\De_n:X\to X^n$ is the diagonal,
then $\de_n=(-1)^{n-1}\De_{n*}$.

\begin{remark}
More generally, let $X$ be a smooth quasi-projective surface equipped with an action of a torus~$T\iso\Gm^n$.
Then by the localization theorem,
the inclusion of the $T$-invariant locus $i:X^T\emb X$
induces an algebra isomorphism
\[i^*:H^*_T(X,\bQ)_\loc\to H^*_T(X^T,\bQ)_\loc
\iso H^*(X^T,\bQ)\ts_\bQ K,
\]
where $K$ is the fraction field of the algebra
\[
A=H^*_T(\pt,\bQ)=H^*(BT,\bQ)\iso S(X^*(T)\ts_\bZ\bQ),\qquad
X^*(T)=\Hom(T,\Gm),
\]
and $H^*_T(X,\bQ)_\loc=H^*_T(X,\bQ)\ts_AK$.
If $X^T$ is projective, then the right-hand side becomes a Frobenius
algebra over $K$ with the counit (trace)
\[
\eps(a)=-\int_{X^T}\frac{a}{\fe(N_{X^T/X})}.
\]
Thus the equivariant intersection theory of the Hilbert schemes of $X$
can be studied using this Frobenius algebra.
The constructions and results below admit corresponding localized equivariant versions.
\end{remark}

\subsection{Tautological bundles}
Given a vector bundle $E$ on $X$,
its \idef{tautological bundle} on $X^{[n]}$ is
\begin{equation}
\hb En=p_{1*}\rbr{\cO_{Z_n}\ts p_2^*E}),
\end{equation}
where $Z_n\sbs X^{[n]}\xx X$ is the universal subscheme
and $p_1,p_2$ are projections.
This is a vector bundle of rank $n\rk(E)$.
On the other hand, for $a\in H$, define
\begin{equation}\label{G}
G(a,n)
=\hb an
=p_{1*}\rbr{\ch(\cO_{Z_n})\cdot p_2^*\rbig{a\td(X)}}\in\bH_n.
\end{equation}
By Grothendieck--Riemann--Roch,
\begin{equation}
\ch(E^{[n]})
=G(\ch(E),n).
\end{equation}
For $a\in H^i$,
let $G_k(a,n)$ denote the component of $G(a,n)$ in $\bH_n^{i+2k}$.
Define the \idef{Chern character operator}
as the cup product operator
\begin{equation}
\fG_k(a)\in\End(\bH),\qquad
\bH_n\ni u\mto G_k(a,n)\cdot u\in\bH_n,\qquad n\ge0.
\end{equation}
We have
\begin{equation}\label{ch as G}
\ch_k(\hb En)
=\sum_{i=0}^k G_{k-i}(\ch_i(E),n),
\end{equation}
hence the cup products with $\ch_k(\hb En)$
can be reduced to the operators $\fG_{k-i}(\ch_i(E))$.

Our ultimate goal is to calculate the integrals
\[
\sum_{n\ge0}q^n\int_{\hb Xn}
\ch_{k_1}(\hb En_1)\dots\ch_{k_r}(\hb En_r)\,\fe_t(T_{\hb  Xn})
=\sum_{n\ge0}q^n\int_{\hb Xn}\Phi(n)\cdot \fe_t(T_{\hb Xn}),
\]
where $\Phi(n)=\ch_{k_1}(\hb En_1)\dots\ch_{k_r}(\hb En_r)\in\bH_n$ for $n\ge0$.
As we will see in \cref{int as trace},
the above expression can be written as a (super) $q$-trace
\[\str\rbr{\qd \W(L,z)\Phi_\bul},\]
where $\Phi_\bul\in\End(\bH)$ is the cup product operator corresponding to
$\Phi=(\Phi(n))_{n\ge0}$
and $\W(L,z)$ is a certain vertex operator.
Therefore we need to understand better the cup products with $\ch_k(\hb En)$
or, equivalently, the Chern character operators $\fG_k(a)$ for $a\in H$.

\subsection{Chern character operators and Nakajima operators}
\label{LQW:sec}
Let $\Zs=\bZ\ms\set0$.
For $\bn=(n_1,\dots,n_r)\in\Zs^r$, define the linear map
\begin{equation}
\fat_\bn
=\fa_{n_1}\ts\dots\ts\fa_{n_r}:H^{\ts r}\to\iend(\bH)^{\ts r}.
\end{equation}
We will compose it with the \NOP \cref{nop}
\[\N:\iend(\bH)^{\ts r}\to \iend(\bH)\]
associated with the degree polarization of $\iend(\bH)$ (see \cref{gr alg pol}).
The composition
\[\N\fat_\bn:H^{\ts r}\to\iend(\bH)\]
can be described as follows.
For $a=a_1\ts\dots\ts a_r\in H^{\ts r}$,
we reorder the factors in $\fat_\bn(a)=\fa_{n_1}(a_1)\ts\dots\ts\fa_{n_r}(a_r)$
so that the degrees are non-increasing.
Since $\deg \fa_{n_i}(a_i)=-n_i$,
we choose $\si\in S_r$
with
$n_{\si (1)}\le\dots \le n_{\si (r)}$.
Then
\begin{equation}
\N\fat_\bn(a_1\ts\dots\ts a_r)
=
\ep(\si,a)
\fa_{n_{\si(1)}}(a_{\si(1)})\dots\fa_{n_{\si(r)}}(a_{\si(r)})
\in\End(\bH),
\end{equation}
where $\ep(\si,a)
=\prod_{i<j,\si(i)>\si(j)}(-1)^{p(a_{\si(i)})p(a_{\si(j)})}$.
Different choices of $\si$ give the same expression,
since $[\fa_n(a_i),\fa_n(a_j)]=0$ (the super-bracket).
We have $\N\fat_{\ta\bn}=\N\fat_\bn\circ\ta\i$ for all $\ta\in S_r$,
where
\[\ta\i(a_1\ts\dots\ts a_r)=\ep(\ta,a)a_{\ta(1)}\ts\dots\ts a_{\ta(r)}.\]

Finally, using the coproduct $\de_r:H\to H^{\ts r}$, we define
\begin{equation}
\fa_\bn=\N\fat_\bn\de_r:H\to\End(\bH).
\end{equation}
Since $\de$ is a (super) cocommutative coproduct, the map $\fa_\bn$ is independent of the order of $\bn$.
Therefore we can define $\fa_\la:H\to\End(\bH)$
for every $\la\in \Zs^r/S_r$, called
a \idef{generalized partition} of length $r$.
We can write $\la$ in the form
\[\la=(\dots(-2)^{m_{-2}}(-1)^{m_{-1}}1^{m_{1}}2^{m_{2}}
\dots)=(i^{m_i})_{i\ne0},\]
where $m:\Zs\to\bN$ is a map with finite support,
such that $m_i$ is the multiplicity of $i$ in $\la$.
Such a map is called a \idef{multiset} on $\Zs$.
Define $\la!=\prod_{i\ne0}m_i!$ and
\begin{equation}
p_k(\la)=\sum_{i\ne0}i^km_i,\qquad
\ell(\la)= p_0(\la)=\sum_{i\ne0}m_i,\qquad
\n\la=p_1(\la)=\sum_{i\ne0}im_i.
\end{equation}

By abuse of notation, we denote the first Chern class
$c_1(K_X)\in H^2$ of the canonical bundle by $K_X$.
Let $e_X=c_2(X)\in H^4$.

\begin{theorem}[See {\cite[Theorem 4.6]{li_hilbert}}]
\label{LQW:th}
Let $a\in H$ satisfy $K_Xa=0$.
For $k\ge1$, we have
\[
\fG_k(a)
=(-1)^{k}\sum_{\ov{\n\la=0}{\ell(\la)=k+2}}
\frac{1}{\la!}\fa_\la(a)
+(-1)^{k-1}
\sum_{\ov{\n\la=0}{\ell(\la)=k}}
\frac{p_2(\la)-2}{24\la!}
\fa_\la(e_Xa).
\]
Moreover,
$\fG_0(a)=\sum_{\n\la=0,\ell(\la)=2}\frac1{\la!}\fa_\la(a)$.
\end{theorem}

\begin{remark}
Our sign discrepancy with \cite{li_hilbert}
stems from the fact that $\de_n=(-1)^{n-1}\De_{n*}$.
\end{remark}

We will show in \cref{univ current} that expressions of the form
\begin{equation}\label{p-expr}
f_{j,r}(a)=\sum_{\ov{\n\la=0}{\ell(\la)=r}}\frac{p_j(\la)}{\la!}\fa_\la(a)
\end{equation}
can be written in terms of vertex operators.
Note that for $j=0$, the expression \cref{p-expr} is
equal to
\[
f_{0,r}(a)=r\sum_{\ov{\n\la=0}{\ell(\la)=r}}
\frac{1}{\la!}\fa_\la(a).\]
Therefore all terms in \cref{LQW:th} are scalar multiples of expressions of the form \cref{p-expr}
with $j=0$ or $j=2$.

\def\bil{\rbr} 

\subsection{Universal bosonic descendant}
\label{univ current}
Let $\F=\Fo(V)$ be the space (polarized algebra) of degree zero fields on~$V$.
For $a\in H$, consider the bosonic descendant \cref{descend}
\eq{\AD_a(z)=\BD1_a(z)=\sum_{n\ne0}a_nz^{-n}\in\F.}

Let $(e_i)_{i\in I}$ and $(f_i)_{i\in I}$
be (homogeneous) dual bases of $H$ \wrt the pairing,
satisfying $(e_i,f_j)=\de_{ij}$.
The \idef{universal bosonic descendant}
\eq{\UD(z)
=\sum_{i\in I} \AD_{e_i}(z)f_i=\sum_{i\in I}\sum_{n\in\Zs} \fa_n(e_i)f_iz^{-n}\in\F\ts H}
is independent of the basis choices,
since the same is true for the Casimir $\sum_{i\in I}e_i\ts f_i$.
For all $a\in H$, we have
\[\bil{a,\UD(z)}=A_a(z).\]

\begin{lemma}
For $a\in H$, we have
\[\de_r(a)
=\sum_{\bi\in I^r}\ep(\bi)\rbr{a,f_{i_1}\dots f_{i_r}}
e_{i_1}\ts\dots\ts e_{i_r},
\]
where $\ep(\bi)=\prod_{k<l}\ep(e_{i_k},e_{i_l})$.
\end{lemma}
\begin{proof}
The product and coproduct on $H$ are compatible \wrt the pairing,
hence
$\rbr{\de_r(a),f_{i_1}\ts\dots\ts f_{i_r}}
=\rbr{a,f_{i_1}\dots f_{i_r}}$.
We have
$(e_{i_1}\ts\dots\ts e_{i_r},f_{i_1}\ts\dots\ts f_{i_r})=\ep(\bi)$,
hence
\[\de_r(a)
=\sum_{\bi\in I^r}\ep(\bi)\rbr{\de_r(a),f_{i_1}\ts\dots\ts f_{i_r}}
e_{i_1}\ts\dots\ts e_{i_r}.
\qedhere
\]
\end{proof}

In what follows, we define the \NOP on $\cF\ts H$
by applying the \NOP on $\cF$ and the cup product on $H$
(with the usual Koszul signs).
For homogeneous $g_i\in\cF$ and $a_i\in H$, we have
\[
\nop{(g_1\ts a_1)\dots(g_r\ts a_r)}
=\prod_{i<j}(-1)^{p(a_i)p(g_j)}
\nop{g_1\dots g_r}\ts a_1\dots a_r.
\]

\begin{lemma}\label{A(z)^r}
For $a\in H$ and $r\ge1$, we have
\[\bil{a,\nop{\UD(z)^r}} =\sum_{\bn\in\Zs^r}\fa_\bn(a)z^{-\n\bn},
\qquad
\n\bn=\sum_{i=1}^rn_i.
\]
\end{lemma}
\begin{proof}
We have
\begin{multline*}
\bil{a,\nop{\UD(z)^r}}
=\sum_{\bn\in\Zs^r}\rbr{
\sum_{\bi\in I^r}\ep(\bi)\bil{a,f_{i_1}\dots f_{i_r}}
\N\rbig{\fa_{n_1}(e_{i_1})\dots\fa_{n_r}(e_{i_r})}
}z^{-\n\bn}\\
=\sum_{\bn\in\Zs^r}\N\fat_\bn
\rbr{
\sum_{\bi\in I^r}
\ep(\bi)\bil{a,f_{i_1}\cdots f_{i_r}}e_{i_1}\ts\dots\ts e_{i_r}
}z^{-\n\bn}\\
=\sum_{\bn\in\Zs^r}
\N\fat_\bn\de_r(a)z^{-\n\bn}
=\sum_{\bn\in\Zs^r}\fa_\bn(a)z^{-\n\bn}.
\tag*{\qedhere}
\end{multline*}
\end{proof}

Every summand $\AD_{e_i}(z)f_i$ of $\UD(z)$
has parity $p(e_i)+p(f_i)=0$, hence $\UD(z)$ is even.
Therefore we can consider the field $e^{t\UD(z)}$.

\begin{theorem}
For every $a\in H$ and $k\ge0$, we have
\[
\bil{a,\nop{{(-D)^k\UD(z)}\cdot e^{t\UD(z)}}}
=\sum_{\ell(\la)\ge1} \frac{p_k(\la)}{\la!}\fa_\la(a)
t^{\ell(\la)-1}z^{-\n\la}.
\]
\end{theorem}

\begin{proof}
We have
$(-D)^k \AD_b(z)=\sum_{n\ne0}n^k\fa_n(b)z^{-n}$.
Applying \cref{A(z)^r}, we obtain
\[
\bil{a,\nop{{(-D)^k\UD(z)}\cdot e^{t\UD(z)}}}
=\sum_{r\ge1}\frac{t^{r-1}}{(r-1)!}
\sum_{\bn\in\Zs^r}n_1^k\fa_{\bn}(a)z^{-\n\bn}.
\]
Let $\la=(i^{m_i})_{i\ne0}$ be a generalized partition
of length $r$.
For a fixed $n\in\Zs$ with $m_n>0$, the number of tuples $\bn\in\Zs^r$ with multiplicities $(m_i)_i$ and first entry $n_1=n$ is
\[
\frac{(r-1)!}{(m_n-1)!\prod_{i\ne n}m_i!}
=\frac{(r-1)!m_n}{\la!}.
\]
Therefore
\[\bil{a,\nop{{(-D)^k\UD(z)}\cdot e^{t\UD(z)}}}
=\sum_{\ell(\la)\ge1}\frac{t^{\ell(\la)-1}}{\la!}
\sum_{n\in\Zs}n^km_n\fa_\la(a)z^{-\n\la}
=\sum_{\ell(\la)\ge1}\frac{p_k(\la)t^{\ell(\la)-1}}{\la!}\fa_\la(a) z^{-\n\la}.
\qedhere
\]
\end{proof}

\begin{corollary}
\label{pk:cor}
For $a\in H$, $k\ge0$ and $r\ge1$, we have
\[
\sum_{\ov{\n\la=0}{\ell(\la)=r}}\frac{p_k(\la)}{\la!}\fa_\la(a)
=(-1)^k[z^0t^{r-1}]\bil{a,\nop{D^k\UD(z)\cdot e^{t\UD(z)}}}
\]
which is a linear combination of expressions
\[[z^0]\nop{D^k\AD_{a_1}(z)\cdot \AD_{a_2}(z)\dots \AD_{a_r}(z)},
\qquad
a_1,\dots,a_r\in H.\]
\end{corollary}

In particular, we have
\begin{gather}
\sum_{\ov{\n\la=0}{\ell(\la)=r}}\frac{\fa_\la(a)}{\la!}
=\frac{1}{r!}[z^0]\bil{a,\nop{\UD(z)^r}},\\
\sum_{\ov{\n\la=0}{\ell(\la)=r}}
\frac{p_2(\la)}{\la!}\fa_\la(a)
=\frac{1}{(r-1)!}
[z^0]\bil{a,\nop{D^2\UD(z)\cdot \UD(z)^{r-1}}}.
\end{gather}

\subsection{Operator \tpdf{$\W$}W}
Given a line bundle $L$ on $X$, let
\[\sE=\sE(m,n)\in K(\hb Xm\xx \hb Xn)\]
be the class (virtual bundle) such that its fiber over
$(\eta,\ksi)\in \hb Xm\xx \hb Xn$ is
\[\sE|_{(\eta,\ksi)}
=\hi(\cO_X,L)-\hi(I_\ksi,I_\eta\ts L),
\]
where $I_\ksi\sbs\cO_X$ is the ideal sheaf of $\ksi\in\hb Xn$
(see \cite{carlsson_vertex,carlsson_exts}).
In particular, if $L=\cO_X$,
then $\De_n^*\sE(n,n)\iso T_{\hb Xn}$
for the diagonal $\De_n:\hb Xn\to\hb Xn\xx\hb Xn$.
More generally, if $L=\cO_X\ts e^t$
is the trivial line bundle equipped with the $\Gm$-action such that $c_1(L)=t\in X^*(\Gm)\iso\bZ$,
then
\begin{equation}\label{E on diag}
\De_n^*\sE(n,n)\iso T_{\hb Xn}\ts e^t.
\end{equation}

Define the even operator $\W(L,z):\bH\to \bH\lser z$ by
\eq{\bil{\W(L,z)a,b}
=z^{n-m}\int_{\hb Xm\xx\hb Xn}(a\ts b)\cdot \fe(\sE(m,n)),\qquad
a\in\bH_m,\, b\in\bH_n,
}
where $\fe=c_{\top}$ is the Euler class.
Note that $\W(L,z)a\in z^{-m}\bH\pser z$.
Writing $\W(L,z)=\sum_{n\in\bZ} f_n z^{n}$,
we obtain $f_{n-m}(a)\in\bH_n$, hence $\deg f_{n-m}=n-m$ and
$\W(L,z)$ is a field of degree zero.

The importance of the operator $\W$ stems from the following observation (\cf \cite[\S2.1.1]{carlsson_exts}).

\begin{lemma}\label{int as trace}
Let $\Phi=(\Phi(n))_{n\ge0}\in\prod_{n\ge0}\bH_n$
and let $\Phi_\bul\in\End(\bH)$ be the corresponding cup product operator.
Then
\[\sum_{n\ge0}q^n\int_{\hb Xn}\Phi(n)\cdot \fe_t(T_{\hb Xn})
=\str\rbr{\qd \W(L,z)\Phi_\bul}
=\str\rbr{\qd \Phi_\bul \W(L,z)},
\]
where $L=\cO_X\ts e^t$
and $\fe_t(F)=\fe(F\ts e^t)=\prod_{i}(t+x_i)$ for a vector bundle $F$ with the Chern roots $(x_i)_i$.
\end{lemma}

\begin{proof}
We can assume that $\Phi(n)\in\bH_n^{\b0}$, since the contribution of odd components is zero.
The operators $q^\bd$ and $\Phi_\bul$ commute,
hence
$\str\rbr{\qd \W\Phi_\bul}
=\str\rbr{\Phi_\bul \qd \W}
=\str\rbr{\qd \Phi_\bul \W}
.$
To determine the trace we can restrict to $[z^0]\W(L,z)\in\End(\bH)$
which has degree zero.

Let $A=H^*(\hb Xn,\bQ)=\bH_n$ with its Frobenius structure $(\mu,\eta,\de,\eps)$
and the pairing
\[\eps\mu:A\ts A\to\bQ,\qquad a\ts b\mto\int_{\hb Xn}ab.\]
Let $u=\Phi(n)$,
$f=[z^0]\W(L,z)|_{\bH_n}$ and $C_f=\fe(\sE)\in A\ts A$
(\cf \cref{repr ker}).
\def\u{u}
For $a,b\in A$, we have
\[(f\Phi_\bul a,b)=(\u a\ts b,C_f)=(a\ts b,(\u\ts 1)C_f).\]
Therefore $(\u\ts 1)C_f$ is the reproducing kernel of $f\Phi_\bul$
\cref{repr ker},
hence by \cref{str and C}
\[\str(f\Phi_\bul)=\eps\mu((\u\ts 1)C_f)=\eps(\u\cdot\mu (C_f))
=\int_{\hb Xn}\Phi(n)\cdot \mu (C_f)
.\]
We have $\mu (C_f)=\De_n^* \fe(\sE)
=\fe_t(T_{\hb Xn})$ by \cref{E on diag}.
\end{proof}

Crucially, the operator $\W$ is a vertex operator.
Let $\fa_n(L)=\fa_n(a)$ for $a=c_1(L)$.

\begin{theorem}
[See {\cite[Theorem 1]{carlsson_vertex}}
and {\cite[Theorem 1]{carlsson_exts}}]
We have
\[\W(L,z)
=\exp\rbr{\sum_{n>0}\frac{z^n}n\fa_{-n}(L-K_X)}
\exp\rbr{\sum_{n>0}\frac{z^{-n}}n\fa_n(-L)}
\]
\end{theorem}

\begin{remark}
Our discrepancy with the signs in \cite{carlsson_vertex}
stems from the fact that our Nakajima operators are related to
\cite{carlsson_vertex} via
$\fa_{-n}=\fa_{-n}^C$ and $\fa_n=-\fa^C_n$ for $n>0$.
\end{remark}

The creation and annihilation parts
of the free boson field $\BF_a(z)$ \cref{free bos} are
\begin{equation}
\BF_a(z)_+=\sum_{n>0}\frac{z^n}n a_{-n},\qquad
\BF_a(z)_-=\sum_{n>0}\frac{z^{-n}}{-n}a_{n}.
\end{equation}
Writing $\BF_L(z)=\BF_a(z)$ for $a=c_1(L)$, we obtain
\begin{equation}
\W(L,z)
=\exp\rbr{\BF_{L-K_X}(z)_+}\cdot \exp\rbr{\BF_L(z)_-}.
\end{equation}
In particular, for $K_X=0$, we obtain
\begin{equation}\label{W for K=0}
\W(L,z)=e^{\BF_{L}(z)_+} e^{\BF_L(z)_-}
=\nop{e^{\BF_L(z)}}.
\end{equation}
If $L=\cO_X\ts e^t$, then $a=c_1(L)=t$
(in $H\ts_\bQ\bQ[t]$)
and $(a,a-K_X)=0$.
Therefore \cref{ta(f)} implies
\eq{\ta(\W)=\frac{\str(q^\d \W)}{\str(q^\d)}=1.}


\subsection{Summary}
To summarize,
by \cref{int as trace}, we have
\[\ang{\ch_{k_1}^{E_1}\dots\ch_{k_r}^{E_r}}
=\sum_{n\ge0}q^n\int_{\hb Xn}\prod_{i=1}^r\ch_{k_i}(\hb{E_i}n)
\cdot\fe_t(T_{\hb Xn})\\
=\str_q\rbr{\W(L,z)\prod_{i=1}^r\bch_{k_i}^{E_i}},
\]
where $L=\cO_X\ts e^t$ and $\bch_k^E\in\End(\bH)$ denotes the cup product operator
corresponding to $\ch_k(\hb En)\in\bH_n$ for $n\ge0$.
We want to prove quasi-modularity of the normalized $q$-trace
\[\ang{\ch_{k_1}^{E_1}\dots\ch_{k_r}^{E_r}}'
=\ta\rbr{\W(L,z)\prod_{i=1}^r\bch_{k_i}^{E_i}}.
\]
Because of \cref{ch as G},
we can consider instead, for $a_1,\dots,a_r\in H$
and $k_1,\dots,k_r\ge0$,
the generating series
\[
\sum_{n\ge0}q^n\int_{\hb Xn}\prod_{i=1}^rG_{k_i}(a_i,n)
\cdot\fe_t(T_{\hb Xn})
=\str_q\rbr{\W(L,z)\prod_{i=1}^r\fG_{k_i}(a_i)}.
\]
and the normalized $q$-trace
\begin{equation}
\ta\rbr{\W(L,z)\prod_{i=1}^r\fG_{k_i}(a_i)}.
\end{equation}
For $K_X=0$, we have seen that
$\W(L,z)$ can be expressed in terms of the free boson field~$\BF_L(z)$.
On the other hand, the operators $\fG_k(a)$ can be expressed in terms of descendants of free boson fields by
\cref{LQW:th} and \cref{pk:cor}.
Therefore our ultimate goal is to prove quasi-modularity of normalized $q$-traces of operators expressible in terms of free boson fields and their descendants.

\section{Constant term theorem}

\subsection{Special functions}
Given a domain $D$ (open connected subset of $\bC^n$),
let $\cO(D)$ denote the ring of holomorphic functions on $D$ and let  $\cM(D)$ denote the field of meromorphic functions on $D$, which is the field of fractions of $\cO(D)$.
Consider the domains in $\bC^{r+1}$
\begin{gather}
X_r=(\bC^\xx)^r\xx\sets{q\in\bC}{0<|q|<1},\\
U_r=\sets{(z_1,\dots,z_r,q)\in \bC^{r+1}}
{1>\n{z_1}>\cdots>\n{z_r}>\n q>0}\sbs X_r.
\end{gather}

Define the $q$-Pochhammer symbol (we use formal expansion)
\begin{equation}\label{poch}
\ph z=\prod_{i\ge0}(1-q^iz)
=\exp\rbr{-\sum_{n>0}\frac1n\frac{z^n}{1-q^n}}.
\end{equation}

Define the theta function
\begin{equation}
\Te(z,q)=\frac{\ph z\ph{z\i q}}{\ph q^2}
\in\cO(X_1),
\end{equation}
which has zeros only at $z=q^n$, $n\in\bZ$.
Using \cref{poch},
we obtain
\begin{gather}
\log\Te(z,q)=\Phi_0(z,q)+2F_0(q),\\
\Phi_0(z,q)
=-\sum_{n>0}\frac1n\frac{z^n+q^nz^{-n}}{1-q^n},\qquad
F_0(q)=\sum_{n>0}\frac1n\frac{q^n}{1-q^n}.
\end{gather}
More generally, define (using $D=z\dd_z$)
\begin{gather}
\label{Phi_k}
\Phi_k(z,q)=D^k\Phi_0(z,q)
=D^{k}\log\Te(z,q)\qquad k\ge1,\\
\label{F_k}
F_k(q)=\sum_{n>0}n^{k-1}\frac{q^n}{1-q^n}
=\sum_{d,n>0}n^{k-1}q^{nd},\qquad k\ge0.
\end{gather}

The function
\[\Phi_1=D\log\Te=D\Te/\Te
\in\cM(X_1)\]
has poles only at $z=q^n$ and the same applies to $\Phi_k=D^{k-1}\Phi_1$, $k\ge1$.
Moreover,
\begin{gather}
\Te(qz,q)=-z\i\Te(z,q)=\Te(z\i,q),\\
\Phi_1(qz,q)=\Phi_1(z,q)-1=-\Phi_1(z\i,q),\\
\Phi_k(qz,q)=\Phi_k(z,q)
=(-1)^{k}\Phi_k(z\i,q),\qquad k\ge2.
\label{Phi(1/z)}
\end{gather}

Define the Eisenstein series (for even $k\ge1$)
\begin{equation}
E_k(q)=1-\frac{2k}{B_k}\sum_{n,d\ge1}n^{k-1}q^{nd}
=1-\frac{2k}{B_k}F_k(q),
\end{equation}
where $B_k\in\bQ$ are Bernoulli numbers.
The algebra of quasi-modular forms
\begin{equation}
\cC_0=\QM(1,\bQ)
=\bQ[E_2,E_4,E_6]\sbs\bQ\pser q,
\end{equation}
contains all $E_{2k}$ for $k\ge1$, hence $F_{2k}\in \cC_0$ for $k\ge1$.
The Eisenstein series are holomorphic on the unit disc,
hence $\cC_0\sbs \cO(X_0)$.

\begin{lemma}\label{Phi coeff}
We have $\Te(e^s,q)\in s \cC_0\pser s$ and
$\Phi_k(e^s,q)\in \cC_0\lser s$ for $k\ge1$.
\end{lemma}
\begin{proof}
We have $\log\frac1{(z;q)_\infty}=\sum_{n\ge1}\frac1n\frac{z^n}{1-q^n}$, hence
\[\log\frac{\Te(z,q)}{1-z}
=\sum_{n\ge1}\frac1n\frac{(2-z^n-z^{-n})q^n}{1-q^n}.
\]
Using $z=e^s$,
the coefficient of $s^{k}/k!$ is zero for odd $k$
and it is
\[-2\sum_{n\ge1}\frac{n^{k-1}q^n}{1-q^n}
=-2F_k\in\cC_0\]
for even $k\ge2$.
Therefore $\frac{\Te(e^s,q)}{1-e^s}\in \cC_0\pser s$, hence $\Te(e^s,q)\in s \cC_0\pser s$.
Using $D=z\dd_z=\dd_s$,
we obtain $\Phi_1(e^s)=D\Te(e^s)/\Te(e^s)\in\cC_0\lser s$.
Therefore $\Phi_k(e^s)=D^{k-1}\Phi_1(e^s)\in\cC_0\lser s$ for $k\ge1$.
\end{proof}

\begin{remark}
We assign weight $2k$ to $E_{2k}$
and equip the algebra $\cC_0$ with the increasing weight filtration $(W_m\cC_0)_{m\in\bZ}$, where $W_m\cC_0$ consists of quasi-modular forms whose homogeneous components have weight at most $m$.
In particular, $F_{2k}\in W_{2k}\cC_0$.
If $A$ is an algebra equipped with an increasing filtration $(W_mA)_{m\in\bZ}$,
we equip the algebra $B=A\lser s$ with the increasing filtration
\[W_mB=\sets{\sum*_{i}f_is^i}{f_i\in W_{m+i}A},\qquad m\in\bZ,\]
so that $s$ has weight $-1$.
We have
\[
\log\frac{\Te(e^s,q)}{1-e^s}
=-2\sum_{\ov{m\ge2}{2\mid m}}F_m(q)\frac{s^m}{m!}.
\]
Since $\dd_s^k\log(1-e^s)\in\bQ\lser s$ has order $-k$,
applying $\dd_s^k$ and taking the coefficient of $s^l$ gives
\eq{\label{weights of Phi_k}
[s^l]\Phi_k(e^s,q)\in W_{k+l}\cC_0,\qquad k\ge1,\, l\in\bZ.}
This implies that $\Phi_k(e^s,q)\in\cC_0\lser s$ has weight $\le k$.
\end{remark}

\subsection{Special algebras}
For $i\in[r]$, define the operators on $\cM(X_r)$
\[T_if=f(z_1,\dots,qz_i,\dots,z_r).\]
A function $f$ is called elliptic if $T_if=f$ for all $i$.
Note that $\Phi_k(z_i)$ and $\Phi_k(z_i/z_j)$ are elliptic for $k\ge2$, while $T_i\Phi_1(z_j)=\Phi_1(z_j)-\de_{ij}$.
Define the quasi-elliptic algebra
\begin{equation}\label{Cr}
\begin{aligned}
\cC_r=\cC_r(z_0,\dots,z_r)
&=\cC_0\sbr{\Phi_k(z_j/z_i)\col 0\le i\ne j\le r,\,k\ge1}\\
&\iso \cC_0\sbr{\Phi_k(z_i),\Phi_k(z_j/z_i)\col 1\le i\ne j\le r,\,k\ge1}
\sbs\cM(X_r),
\end{aligned}
\end{equation}
where the isomorphism is obtained by fixing $z_0=1$.
Similarly, define the elliptic algebra
\[\cE_r
=\cC_0\sbr{\Phi_k(z_i),\Phi_k(z_j/z_i)\col
1\le i\ne j\le r,\,k\ge2}\sbs \cM(X_r).\]
In particular, $\cE_0=\cC_0$.
Note that it is enough to consider only $\Phi_k(z_j/z_i,q)$ for $i<j$ because of~\cref{Phi(1/z)}.
The operators $T_i$ preserve $\cC_r$.

We equip the algebra $\cC_r$ with the increasing weight filtration
$(W_m\cC_r)_{m\in\bZ}$, where $W_m\cC_r=0$ for $m<0$,
while $E_{2k}$ has weight $2k$ and
$\Phi_k(z_j/z_i)$ has weight $k$.
The operators $T_i$ preserve the weight filtration.

\begin{lemma}
\label{T is surj}
The map $T_i-1:W_{m+1}\cC_r\to W_m\cC_r$ is surjective for all $m\ge0$.
In particular, $T_i-1:\cC_r\to\cC_r$ is surjective.
\end{lemma}
\begin{proof}
Let $B\subset\cC_r$ be the subalgebra generated by $\cE_r$,
the functions $\Phi_1(z_j)$ for $j\ne i$,
the functions $\Phi_1(z_i/z_j)-\Phi_1(z_i)$ for $j\ne i$,
and the functions $\Phi_1(z_j/z_k)$ for $j,k\ne i$ with $j\ne k$.
Then every element of $B$ is $T_i$-invariant.

Let $g=-\Phi_1(z_i)$.
Then $T_ig=g+1$ and $\cC_r=B[g]$.
Since the change of basis from $1,g,g^2,\dots$ to
$1,\binom g1,\binom g2,\dots$ is triangular and preserves the weight filtration, every $f\in W_m\cC_r$ can be written as
\[f=\sum_{k\ge0}b_k\binom{g}{k},\qquad
b_k\in B\cap W_{m-k}\cC_r.
\]
Then
\[h=\sum_{k\ge0}b_k\binom{g}{k+1}\in W_{m+1}\cC_r.
\]
Using $\binom{x+1}{k+1}-\binom{x}{k+1}=\binom{x}{k}$,
we obtain
\[(T_i-1)h
=\sum_{k\ge0}b_k\rbr{\binom{g+1}{k+1}-\binom{g}{k+1}}=f.
\qedhere
\]
\end{proof}

The poles of the functions in $\cC_r$ occur only on the divisors
$z_j/z_i=q^n$ for $0\le i\ne j\le r$.
Therefore
\[\cC_r\sbs\cO(U_r)\cap\cM(X_r)\sbs\cM(U_r).\]
The function $\Phi_1(z,q)$ has an expansion
\begin{equation}
\Phi_1(z,q)
=\sum_{n>0}\frac{q^nz^{-n}-z^n}{1-q^n}
\in\bQ(z)\pser q.
\end{equation}
The same applies to $\Phi_k(z,q)$ for $k\ge2$.
Since $\cC_r$ is generated by $\Phi_k(z_j/z_i)$,
we obtain $\cC_r\sbs\bQ(z_1,\dots,z_r)\pser q$.
There is an expansion map
(corresponding to $1>\n {z_1}>\dots>\n{z_r}>\n q>0$)
\begin{equation}
\io_r:\cC_r\emb\cL_r:=\bQ\lser{z_1}\dots\lser{z_r}\pser q.
\end{equation}

\subsection{Constant term theorem}
Define the constant term of $f\in\cC_r$ to be
\begin{equation}
\CT_r(f)=[z_1^0\dots z_r^0]\io_r(f).
\end{equation}
The goal of this section is to prove the following result.

\begin{theorem}[Constant term theorem]
\label{CTT}
If $F\in W_{m}\cC_r$, then
\[\CT_r(F)\in W_{m}\cC_0.\]
In particular, if $F\in\cC_r$, then $\CT_r(F)\in\cC_0$.
\end{theorem}

A closely related result appears in \cite[Theorem 5.8]{goujard_counting}.

Recall that
\[U_r=\sets{(z_1,\dots,z_r,q)\in \bC^{r+1}}
{1>\n{z_1}>\cdots>\n{z_r}>\n q>0}.\]
We will introduce the coefficient map $[z_r^0]:\cO(U_r)\to\cO(U_{r-1})$
and show that $\cC_r$ is mapped to~$\cC_{r-1}$.
Applying this operation $r$ times, we will obtain
$\CT_r(f)\in\cC_0$.
\medskip

In what follows we will write $x=z_r$ and $z_0=1$.
The fiber of the projection $U_r\to U_{r-1}$
over $u=(z_1,\dots,z_{r-1},q)\in U_{r-1}$ is the annulus
\[
A_u=\sets{x\in\bC}{\n q<\n x<\n{z_{r-1}}}.
\]
For every $f\in\cO(U_r)$, the restriction $f(u,-)$
is holomorphic on $A_u$, and therefore has a unique Laurent expansion
\[
f(u,x)=\sum_{n\in\bZ}a_n(u)x^n,\qquad x\in A_u.
\]
The coefficients $a_n(u)$ are holomorphic on $U_{r-1}$.
Thus one has well-defined linear maps
\[
[x^n]:\cO(U_r)\longrightarrow\cO(U_{r-1}),\qquad
[x^n]f:=a_n.
\]

Let $a:X_{r-1}\to\bC$ be a holomorphic map such that its graph
\[D_a=\sets{(u,a(u))}{u\in X_{r-1}}\]
is contained in $X_r$.
For \(f\in\cM(X_r)\), write locally near \(D_a\)
\[
\frac{f(u,x)}{x}
=
\sum_{n\gg-\infty} f_n(u)\bigl(x-a(u)\bigr)^n.
\]
Then define
\[
\Res_{x=a} f\,\frac{dx}{x}:=f_{-1}\in\cM(X_{r-1}).
\]

\begin{lemma}\label{residue ident}
For $f\in\cC_r$ we have
an equality in $\cM(U_{r-1})$
\[
[x^0](1-T_r)f
=\sum_{i=0}^{r-1}\Res_{x=qz_i}f\,\frac{dx}{x}.
\]
\end{lemma}
\begin{proof}
Let $u=(z_1,\dots,z_{r-1},q)\in U_{r-1}$.
Choose $\n q<R<\n{z_{r-1}}$, so that
\[\setc x{\n x=R}\sbs A_u.\]
Since $f\in\cC_r\subset\cO(U_r)$,
the function $f(u,-)$ is holomorphic on $A_u$,
hence
\[
\frac{1}{2\pi i}\int_{\n x=R}f(u,x)\frac{dx}{x}
=[x^0]f(u,-).
\]
Likewise, since $T_rf\in\cC_r$,
\[
\frac{1}{2\pi i}\int_{|y|=R}(T_rf)(u,y)\frac{dy}{y}
=
[x^0]T_rf(u,-).
\]
After the change of variables \(x=qy\), we obtain
\[
\frac{1}{2\pi i}\int_{|y|=R}(T_rf)(u,y)\frac{dy}{y}
=
\frac{1}{2\pi i}\int_{|x|=|q|R}f(u,x)\frac{dx}{x}.
\]
Therefore
\[
[x^0](1-T_r)f(u)
=\frac{1}{2\pi i}\int_{|x|=R}f(u,x)\frac{dx}{x}
-\frac{1}{2\pi i}\int_{|x|=|q|R}f(u,x)\frac{dx}{x}.
\]
As a meromorphic function of $x$,
the only possible poles of $f(u,x)$ are $x=q^nz_i$ for $0\le i<r$.
If $x=q^nz_i$ is in the annulus $|q|R\le |x|\le R$, then
$\n q^2<\n q^n\n {z_i}<\n {z_i}$, hence $n=1$.
The points $x=qz_i$ lie in the interior of the annulus.
Now we apply the residue theorem.
\end{proof}

\begin{proposition}[Residue closure]
\label{resid closure}
Let $f\in W_m\cC_r$ and let $\set{x=a}$ be a pole divisor of $f$.
Then
\[
\Res_{x=a}f\,\frac{dx}{x}\in W_{m-1}\cC_{r-1}.
\]
\end{proposition}

\begin{proof}
Let $a=q^nz_j$ for $n\in\bZ$, $0\le j<r$ and let $x=ae^s$, where $s$ is a local coordinate around the divisor.
Then
\[\Res_{x=a}f(x)\frac{dx}x=\Res_{s=0}f(ae^s)ds=[s\i]f(ae^s).\]
It suffices to show that
$[s^\ell]f(ae^s)\in W_{m+\ell}\cC_{r-1}$ for all $\ell\in\bZ$
or, equivalently, $f(ae^s)\in W_m(\cC_{r-1}\lser s)$.
It is enough to check this for the generators of $\cC_r$,
by the multiplicativity of the filtration.
Generators which do not involve $x$ are independent of $s$.

By \cref{Phi(1/z)}, among generators involving $x$,
it is enough to consider generators
of the form $\Phi_k(x/z_i)$.
If $i\ne j$, then
\[
\Phi_k(q^ne^sz_j/z_i)
=\Phi_k(e^sz_j/z_i)-n\delta_{k,1}
=\sum_{\ell\ge0}
\frac{\Phi_{k+\ell}(z_j/z_i)}{\ell!}s^\ell -n\delta_{k,1}.
\]
Hence the coefficient of $s^\ell$ belongs to $W_{k+\ell}\cC_{r-1}$.
If $i=j$, then
$\Phi_k(q^ne^s)=\Phi_k(e^s)-n\delta_{k,1}$.
The coefficient of $s^\ell$ belongs to
$W_{k+\ell}\cC_{r-1}$ by \cref{weights of Phi_k}.
\end{proof}

\begin{proof}[Proof of \cref{CTT}]
Let $f\in W_{m}\cC_r$.
By \cref{T is surj}, there exists $g\in W_{m+1}\cC_r$
such that $(1-T_r)g=f$.
By \cref{residue ident},
\[
[z_r^0]f
=
\sum_{i=0}^{r-1}
\Res_{x=qz_i}g\,\frac{dx}{x}.
\]
By \cref{resid closure}, every term on the right belongs to
$W_{m}\cC_{r-1}$, hence $[z_r^0]f\in W_{m}\cC_{r-1}$.
The ordered expansions are compatible with taking the last constant
coefficient:
\[
\io_{r-1}\rbig{[z_r^0]f}
=
[z_r^0]\io_r(f).
\]
Iterating this argument for $z_{r-1},\dots,z_1$, we obtain
$\CT_r(f)=[z_1^0\dots z_r^0]\io_r(f)\in W_{m}\cC_0$.
\end{proof}

\section{Quasi-modularity of \tpdf{$q$}{q}-traces}
In this section we work, for simplicity, with ordinary vector spaces.
The superspace version needed for the Hilbert scheme application is
recorded in \cref{rem:super qm}.

\subsection{Cumulants}
Let $V\in\vect_K^\bZ$ be a bounded below vector space and let $A=\iend(V)$.
In \cref{norm q-trace} we constructed the normalized $q$-trace $\ta:A\to K\lser q$,
which intuitively may be viewed as the expectation on the algebra of observables $A$.

More generally, let~$A$ be a polarized algebra
and let
\[\ta:A\to B\]
be a linear map to a commutative algebra $B$ such that $\ta(1)=1$.
Given a tuple $a=(a_1,\dots,a_r)\in A^r$, consider the series
\begin{equation}
Z_a(t)=Z_a(t_1,\dots,t_r)
=\tau\rbr{\prod*_{i=1}^r\nop {e^{t_ia_i}}}
\in B\pser{t_1,\dots,t_r}
\end{equation}
and define the cumulants
\begin{gather}
\kp_m(a)=\dd_{t_1}^{m_1}\dots \dd_{t_r}^{m_r}\log Z_a(t)\big|_{t=0}\in B,\qquad m\in\bN^r,\\
\log Z_a(t)=\sum_{m\in\bN^r}\frac{\kp_m(a)}{\prod_i m_i!}t^{m}.
\end{gather}

\begin{example}
For $a,b\in A$, we have
\[\kp_1(a)=\ta(a),\qquad \kp_{1,1}(a,b)=\ta(ab)-\ta(a)\ta(b),\qquad
\kp_2(a)=\ta(\nop{a^2})-\ta(a)^2.
\]
\end{example}

\begin{definition}
A subspace $U\sbs A$ is called \idef{Gaussian} if,
for all $a\in U^r$ and $m\in\bN^r$ with $\n m=\sum*_i m_i\ge3$,
we have $\kp_m(a)=0$.
This means that $\log Z_a(t)$ is a polynomial of degree $\le 2$ for all $a\in U^r$.
\end{definition}

\subsection{Cumulants of linear Heisenberg fields}
Let $H$ be a quadratic space over a field $K$.
Let $\hh$ be the corresponding Heisenberg algebra and $V=V(H)$ be the Fock space.
For every $n\in\bZ$, consider the linear map
$\rho_n:Ht^n\to\End(V)$ (see \cref{rho}).

\begin{definition}
A \idef{linear Heisenberg field (of degree $0$}) is a field
\[
X(z)=\sum*_{n\in\bZ}x(n)z^{-n}\in\Fo(V)\iso\prod*_{n\in\bZ}\iend_n(V)
\]
such that $x(n)\in \Im(\rho_n)$ for all $n$.
Let $\Ho(V)\sbs\Fo(V)$ be the space of all such fields.
\end{definition}

Given linear Heisenberg fields $X_1,\dots,X_r\in\Ho(V)$,
by \cref{exp field} we can construct the vertex operator
\[\nop{e^{X_1(z_1)}}\dots \nop{e^{X_r(z_r)}}
\in \hom_0(V,V\lser {z_1}\dots \lser{z_r}).
\]
Since $V$ is bounded below and has finite-dimensional graded components, we may consider the normalized $q$-trace \cref{norm q-trace}
\[
\ta\rbr{\prod_{i=1}^r\nop{e^{X_i(z_i)}}}
\in K\lser {z_1}\dots \lser{z_r}\pser q.
\]

\begin{theorem}\label{th-gauss2}
Let $X_i(z)=\sum_{n\in\bZ} x_i(n)z^{-n}$, $i\in[r]$,
be linear Heisenberg fields.
Then
\begin{gather}
\log\ta\rbr{\prod_{i=1}^r\nop{e^{t_iX_i(z_i)}}}
=\sum_{n>0}\sum_{i<j}t_it_jc_{ij}^{(n)}
+\sum_{n>0}\sum_{i,j}\frac{q^n}{1-q^{n}}t_it_jc_{ij}^{(n)},\\
\pr cn_{ij}=[x_i(n),x_j(-n)]\cdot (z_j/z_i)^n,\qquad
i,j\in[r],\,n>0. \label{cn}
\end{gather}
In particular,
the space $\Ho(V)$ of linear Heisenberg fields is Gaussian.
\end{theorem}
\begin{proof}
For $N\ge1$, truncate each $X_i$ to the modes $\n n\le N$.
The computation below proves the asserted identity for the truncated fields.
Passing to the limit $N\to\infty$ coefficientwise in the formal variables gives the general case.
Let
\[c_{ij}=[X_i(z_i)_-,X_j(z_j)_+]
=\sum_{m,n>0}[x_i(m)z_i^{-m},x_j(-n)z_j^n]
=\sum_{n>0}c_{ij}^{(n)}.\]
Then
\[e^{t_iX_i(z_i)_-}e^{t_jX_j(z_j)_+}
=e^{t_it_jc_{ij}}\cdot e^{t_jX_j(z_j)_+}e^{t_iX_i(z_i)_-}.
\]
Therefore
\def\aa#1#2{X_{#1,#2}}
\[\prod_{i=1}^r \nop{e^{t_iX_i(z_i)}}
=\exp\rbigg{\sum_{i<j}t_it_jc_{ij}}
e^{\sum_it_iX_i(z_i)_+}e^{\sum_i t_i X_i(z_i)_-}
=\exp\rbr{\sum_{i<j}t_it_jc_{ij}}\prod_{n>0}e^{x_n}e^{y_n},
\]
where $x_n=\sum_{i=1}^r t_ix_{i}(-n)z_i^n$ and $y_n=\sum_{i=1}^rt_ix_{i}(n)z_i^{-n}$.
For $m\ne n$, the elements $x_m,y_m$ commute with $x_n,y_n$,
while $[y_n,x_n]=\sum_{i,j}t_it_jc_{ij}^{(n)}$.
Applying \cref{ta-exp} successively to the factors $e^{x_n}e^{y_n}$,
we obtain
\[\ta\rbr{\prod_{n>0}e^{x_n}e^{y_n}}
=\prod_{n>0}\exp\rbr{\frac{q^n\sum_{i,j}t_it_j c_{ij}^{(n)}}{1-q^n}}
=\exp\rbr{\sum_{n>0}\sum_{i,j}\frac{q^n}{1-q^n}t_it_jc_{ij}^{(n)}}.
\]
Therefore
\[\log \ta\rbr{\prod_{i=1}^r\nop{e^{t_iX_i(z_i)}}}
=\sum_{i<j}t_it_jc_{ij}
+\sum_{n>0}\sum_{i,j}\frac{q^n}{1-q^n}t_it_jc_{ij}^{(n)}.
\qedhere
\]
\end{proof}

\begin{corollary}\label{cumul-cor}
Let $X_1,\dots,X_r$ be linear Heisenberg fields.
Then
\[
\log\ta\rbr{\prod_{i=1}^r\nop{e^{t_iX_i(z_i)}}}
=\sum_{i<j}\kp_{11}(X_i,X_j)t_it_j
+\oh\sum_{i}\kp_2(X_i)t_i^2,
\]
where (see \cref{cn})
\begin{gather}
\kp_{11}(X_1,X_2)
=\sum_{n>0}\frac{\pr cn_{12}+q^n\pr cn_{21}}{1-q^n}
,\qquad
\kp_2(X_1)
=2\sum_{n>0}\frac{q^n\pr cn_{11}}{1-q^n}
\end{gather}
\end{corollary}
\begin{proof}
For $Z(t)=\ta\rbr{\nop{e^{t_1X_1(z_1)}} \nop{e^{t_2X_2(z_2)}}}$,
we have
\[\kp_{11}(X_1,X_2)=[t_1t_2]\log Z(t),\qquad
\kp_{2}(X_1)=2[t_1^2]\log Z(t).
\]
Therefore
$\kp_{11}(X_1,X_2)
=\sum_{n>0}\rbr{\pr cn_{12}+\frac{q^n}{1-q^n}\rbr{\pr cn_{12}+\pr cn_{21}}}=\sum_{n>0}\frac{\pr cn_{12}+q^n\pr cn_{21}}{1-q^n}$.
\end{proof}

\subsection{Cumulants of bosonic descendants}
Let $H$ be a quadratic space, $\hh$ be the Heisenberg algebra
and $V=V(H)$ be the Fock space.
For $a\in H$, consider the free boson field \cref{free bos}
\begin{equation}
\BF_a(z)=\sum_{n\ne0}\frac{a_{-n}}nz^n
\end{equation}
and its descendants $\BD k_a(z)=D^k\BF_a(z)$, where $D=z\dd_z$.

\begin{lemma}\label{cum des}
We have
\begin{gather}
\kp_{11}(\BD k_a(z_1),\BD l_b(z_2))
=(-1)^k(a,b)\Phi_{k+l}(z_2/z_1,q),\qquad k,l\ge0,\\
\kp_2(\BD k_a(z))=2(-1)^{k-1}(a,a)F_{2k}(q),\qquad k\ge0,
\end{gather}
where $\Phi_k$ and $F_k$ are functions defined in \cref{Phi_k,F_k}.
\end{lemma}
\begin{proof}
Let $X_i(z_i)=\sum_n x_i(n)z_i^{-n}=\BF_{a_i}(z_i)$,
where $a_1=a$ and $a_2=b$.
Then
\[\pr cn_{ij}=[x_i(n),x_j(-n)](z_j/z_i)^n
=\sbr{\frac{a_{i,n}}{-n},\frac{a_{j,-n}}n}
(z_j/z_i)^n
=-\frac{(a_i,a_j)}{n}(z_j/z_i)^n.
\]
Using $x=z_2/z_1$, we have by \cref{cumul-cor}
\[
\kp_{11}(\BF_a(z_1),\BF_b(z_2))
=-(a,b)\sum_{n>0}\frac1n\frac{x^n+q^nx^{-n}}{1-q^n}
=(a,b)\Phi_0(x,q).
\]
We have $D_{z_1}=-D_x$ and $D_{z_2}=D_x$.
Therefore
\[\kp_{11}(D^k\BF_a(z_1),D^l\BF_b(z_2))
=(-1)^kD^{k+l}_x\kp_{11}(\BF_a(z_1),\BF_b(z_2))
=(-1)^k(a,b)\Phi_{k+l}(x,q).
\]
Similarly, for $X_1(z)=\sum_n x_1(n)z^{-n}=D^k\BF_a(z)$,
\[\pr cn_{11}
=[x_1(n),x_1(-n)]
=[(-n)^{k-1}a_n,n^{k-1}a_{-n}]
=(-1)^{k-1}n^{2k-1}(a,a).\]
Therefore by \cref{cumul-cor}
\[\kp_2(D^k\BF_a(z))
=2(-1)^{k-1}(a,a)\sum_{n>0}n^{2k-1}\frac{q^n}{1-q^n}
=2(-1)^{k-1}(a,a)F_{2k}(q).
\qedhere
\]
\end{proof}

\begin{example}
We have
\[\log\ta\rbr{\nop{e^{t \BD k_a(z)}}}
=\oh\kp_2(D^k\BF_a(z))t^2
=(-1)^{k-1}(a,a)F_{2k}(q)t^2.\]
In particular, if $(a,a)=0$, then $f(z)=\nop{e^{\BD k_a(z)}}$
satisfies $\ta(f)=1$ (\cf \cref{ta(f)}).
\end{example}

Define the spaces of bosonic descendants
\begin{equation}\label{log des}
\begin{gathered}
\D=\spn\sets{\BD k_a(z)}{a\in H,\, k\ge1}\sbs\Ho(V)\sbs\Fo(V),\\
\Db=\spn\sets{\BD k_a(z)}{a\in H,\, k\ge0}\sbs\Ho(V)\sbs\Fo(V).
\end{gathered}
\end{equation}
We assign weight $k$ to the bosonic descendant $\BD k_a(z)$.


\def\Fil{W}
\begin{lemma}\label{cumul in Cr}
For $k\ge0$, $l\ge1$, and $a,b\in H$, we have
\[
\kp_{11}(\BD k_a(z_0),\BD l_b(z_1))\in W_{k+l}\cC_1(z_0,z_1).
\]
For $X(z)=\sum_{i=1}^m s_i\BD{k_i}_{a_i}(z)$,
where $k_i\ge1$ and $s_i$ are commuting variables,
we have
\[
[s_is_j]\kp_2(X)\in W_{k_i+k_j}\cC_0.
\]
\end{lemma}
\begin{proof}
By \cref{cum des},
\[
\kp_{11}(\BD k_a(z_0),\BD l_b(z_1))
=(-1)^k(a,b)\Phi_{k+l}(z_1/z_0,q)
\in W_{k+l}\cC_1(z_0,z_1).
\]

For $X(z)=\sum_n x(n)z^{-n}=\sum_i s_i\BD{k_i}_{a_i}(z)$,
\cref{cumul-cor} implies
\[\kp_2(X)=2\sum_{n>0}\frac{q^n}{1-q^n}\sum_{i,j}c_{ij}^{(n)},\]
where
\[c_{ij}^{(n)}
=[x_i(n),x_j(-n)]
=(-1)^{k_i-1}s_is_j(a_i,a_j)n^{k_i+k_j-1}.
\]
Therefore
\[
\kp_2(X)
=2\sum_{i,j}(-1)^{k_i-1} s_is_j(a_i,a_j)F_{k_i+k_j}(q).
\]
If $k_i+k_j$ is odd, then the corresponding summands cancel.
Otherwise,
$F_{k_i+k_j}\in W_{k_i+k_j}\cC_0$,
hence $[s_is_j]\kp_2(X)\in W_{k_i+k_j}\cC_0$.
\end{proof}

\subsection{Traces of admissible operators}
In this section we will assume that $H$ is a quadratic space over $K=\bQ$, although all arguments work also for quadratic superspaces.
Let $V=V(H)$ be the Fock space corresponding to $H$.

\begin{definition}\label{adm oper:def}
An \idef{admissible field} is a linear combination of the fields
\begin{equation}
\nop{X_1(z)\dots X_k(z)}\in \Fo(V),
\end{equation}
where $X_i\in \D$.
An \idef{admissible operator} is an operator of the form
\begin{equation}\label{adm oper}
T=[z^0]\cT(z)\in\iend_0(V),
\end{equation}
where $\cT(z)$ is an admissible field.
\end{definition}

We call $\sum_{i=1}^r k_i$ the weight of the normally ordered monomial
\[
\nop{\BD{k_1}_{a_1}(z)\dots \BD{k_r}_{a_r}(z)}.
\]
An admissible field has weight at most $w$ if it is a linear
combination of such monomials of weight at most $w$.
An admissible operator $T=[z^0]\cT(z)$ has weight at most $w$
if $\cT(z)$ has weight at most $w$.

We will study $q$-traces of admissible fields and admissible operators.
However, for our applications to Hilbert schemes,
we need a minor modification of the normalized $q$-trace
$\ta:\iend(V)\to \bQ\pser q$.
Let $X_0\in\Db$ be a fixed even field of weight zero and let
\eq{W(z)=\nop{e^{X_0(z)}}\in\Fo(V).
}
Define the modified $q$-trace
\eq{
\ta_W(f)=\frac{\ta(W(z_0)f)}{\ta(W(z_0))}
\in \bQ\pser q[z_0^{\pm1}],\qquad f\in\iend(V).}
If $f\in\iend_0(V)$, then $\ta_W(f)\in \bQ\pser q$.
Note that $\ta_W=\ta$ for $X_0=0$.

\begin{lemma}\label{gauss-mod trace}
Let $X_1,\dots,X_k\in \Ho(V)$.
Then
\[
\log\ta_W\rbr{\prod_{i=1}^k\nop{e^{t_iX_i(z_i)}}}
=\sum_{i=1}^k t_i\kp_{11}(X_0,X_i)
+\sum_{0<i<j}t_it_j\kp_{11}(X_i,X_j)
+\oh\sum_{i=1}^k t_i^2\kp_2(X_i).
\]
\end{lemma}
\begin{proof}
We have by \cref{cumul-cor} 
\begin{multline*}
\log\ta_W\rbr{\prod_{i=1}^k\nop{e^{t_iX_i(z_i)}}}
=\log\ta\rbr{\nop{e^{t_0X_0(z_0)}}\prod_{i=1}^k\nop{e^{t_iX_i(z_i)}}}
\bigg|_{t_0=1}
-\log\ta\rbr{\nop{e^{t_0X_0(z_0)}}}
\Big|_{t_0=1}
\\
=\rbr{\sum_{i=1}^k t_0t_i\kp_{11}(X_0,X_i)
+\sum_{0<i<j}t_it_j\kp_{11}(X_i,X_j)
+\oh\sum_{i=1}^k t_i^2\kp_2(X_i)}
\bigg|_{t_0=1}.
\tag*{\qedhere}
\end{multline*}
\end{proof}

\begin{theorem}
\label{adm fields:th}
Let $\cT_1,\dots,\cT_r\in\Fo(V)$ be admissible fields
of weights at most $w_1,\dots,w_r$, respectively, and let $w=\sum_i w_i$.
Then
\[\ta_W(\cT_1(z_1)\dots\cT_r(z_r))\in W_{w}\cC_r.
\]
\end{theorem}

\begin{proof}
We can assume that $\cT_i(z)=\nop{X_{i1}(z)\dots X_{ip_i}(z)}$,
where $X_{ij}=\BD{k_{ij}}_{a_{ij}}$ and $\sum_{j=1}^{p_i} k_{ij}\le w_i$.
Let $P=\sets{(i,j)}{i\in [r],\, j\in[p_i]}$.
For $(i,j)\in P$, introduce a variable $s_{ij}$ of degree $k_{ij}$.
Let
\begin{gather*}
Z(\bz,\bs)=\ta_W\rbr{\prod_{i=1}^r\nop{e^{L_i(z_i)}}},\qquad
L_i(z_i)=\sum_{j=1}^{p_i}s_{ij}X_{ij}(z_i).
\end{gather*}
By \cref{gauss-mod trace}, we have
\[\log Z(\bz,\bs)
=\sum_{i=1}^r \kp_{11}(X_0,L_i)
+\sum_{i<j}\kp_{11}(L_i,L_j)
+\oh\sum_{i=1}^r\kp_2(L_i).
\]
By \cref{cumul in Cr},
\begin{enumerate}
\item
the coefficient of $s_{ij}$ in
$\kp_{11}(X_0,L_i)$ belongs to
$W_{k_{ij}}\cC_r$,
\item
the coefficient of $s_{ij}s_{kl}$ in $\kp_{11}(L_i,L_k)$ belongs to
$W_{k_{ij}+k_{kl}}\cC_r$,
\item
the coefficient of $s_{ij}s_{il}$ in $\kp_2(L_i)$ belongs to
$W_{k_{ij}+k_{il}}\cC_0$.
\end{enumerate}
Thus every coefficient of $\log Z(\bz,\bs)$
has weight at most its degree in the variables $s_{ij}$.
The same is therefore true for $Z(\bz,\bs)$.
Therefore
\[
\sbr{\prod*_{(i,j)\in P}s_{ij}}
Z(\bz,\bs)=\ta_W(\cT_1(z_1)\dots\cT_r(z_r))
\]
has weight at most $\sum_{i,j}k_{ij}\le\sum_i w_i=w$.
\end{proof}

\begin{theorem}\label{qm for adm oper:th}
Let $T_1,\dots,T_r\in\iend_0(V)$ be admissible operators
of weights at most $w_1,\dots,w_r$, respectively, and let $w=\sum_i w_i$.
Then
\[
\ta_W(T_1\dots T_r)
\in W_w\cC_0.
\]
\end{theorem}
\begin{proof}
Choose admissible fields $\cT_i(z_i)$ of weight at most $w_i$ such that
$T_i=[z_i^0]\cT_i(z_i)$.
By \cref{adm fields:th},
\[\ta_W(\cT_1(z_1)\dots \cT_r(z_r)) \in W_{w}\cC_r.
\]
Hence, by the constant term Theorem \ref{CTT},
\[
\ta_W(T_1\dots T_r)
=[z_1^0\dots z_r^0]\ta_W(\cT_1(z_1)\dots \cT_r(z_r))\in
W_{w}\cC_0.
\qedhere
\]
\end{proof}

\begin{remark}\label{rem:super qm}
\Cref{adm fields:th,qm for adm oper:th} remain valid for a quadratic
superspace $H$ and its super Fock space, with $q$-traces replaced by
super $q$-traces, provided that the fixed field $X_0$ defining
$W=\nop{e^{X_0}}$ is even.
In the super setting, an admissible field is a finite linear combination of normally ordered monomials
$\nop{X_1(z)\dots X_k(z)}$,
where the $X_i\in\D$ are homogeneous.
In the proof of \cref{adm fields:th}, for every homogeneous $X_{ij}$,
we introduce a supercommuting auxiliary variable $s_{ij}$ with
$p(s_{ij})=p(X_{ij})$.
Then $s_{ij}X_{ij}$ is even, so the exponential computations used above
apply without change. The required normally ordered monomials are
recovered by extracting the coefficients of the corresponding ordered
monomials in the variables $s_{ij}$.
The resulting pairwise contractions are, up to the Koszul signs,
the same scalar functions $\Phi_k$ and $F_{2k}$ as above, multiplied
by the pairing. Hence the proofs of
\cref{adm fields:th,qm for adm oper:th} carry over to the super setting.
\end{remark}

\section{Applications}

\subsection{Application to Hilbert schemes}
\label{app Hilb}
Given vector bundles $E_1,\dots,E_r$ on a smooth projective surface $X$, we consider the generating series
\[
\ang{\ch_{k_1}^{E_1}\dots\ch_{k_r}^{E_r}}
=\sum_{n\ge0}q^n\int_{\hb Xn}
\ch_{k_1}(\hb En_1)\dots\ch_{k_r}(\hb En_r)\, \fe_t(T_{\hb  Xn})
\]
and its reduced (or normalized) form
\[\ang{\ch_{k_1}^{E_1}\dots\ch_{k_r}^{E_r}}'
=\ang{\ch_{k_1}^{E_1}\dots\ch_{k_r}^{E_r}}/\ang{}
=\ph q^{\hi(X)}\ang{\ch_{k_1}^{E_1}\dots\ch_{k_r}^{E_r}}.
\]

\begin{remark}
For $d=\sum_{i=1}^r k_i$, we have
\[\ang{\ch_{k_1}^{E_1}\dots\ch_{k_r}^{E_r}}'
=t^d\ang{\ch_{k_1}^{E_1}\dots\ch_{k_r}^{E_r}}'_{t=1}.
\]
Therefore, in what follows, we assume that $t=1$.
\end{remark}

\begin{theorem}
If the canonical class $K_X$ is numerically trivial, then the series
\[\ang{\ch_{k_1}^{E_1}\dots\ch_{k_r}^{E_r}}'\]
is quasi-modular of weight at most $\sum_i(k_i+2)$.
\end{theorem}
\begin{proof}
By assumption $c_1(K_X)=0$, in $H^2(X,\bQ)$.
Let $\Phi(n)=\ch_{k_1}(\hb En_1)\dots\ch_{k_r}(\hb En_r)\in\bH_n$ for $n\ge0$.
By \cref{int as trace}, we have
\[
\ang{\ch_{k_1}^{E_1}\dots\ch_{k_r}^{E_r}}
=\str\rbr{\qd \W\Phi_\bul}=\str_q(\W\Phi_\bul),
\]
where $\W=\W(L,z)=\nop{e^{\BF_L(z)}}$
by \cref{W for K=0}
for $L=\cO_X\ts e^t$.

Let $\ta(f)=\frac{\str_q(f)}{\str_q(1)}$ be the normalized (super) $q$-trace.
Since $a=c_1(L)=t$ satisfies $(a,a)=0$,
we have $\ta(\W)=1$ by \cref{ta(f)},
hence $\ta_\W(f)=\frac{\ta(\W f)}{\ta(\W)}=\ta(\W f)$.
Therefore
\[
\ang{\ch_{k_1}^{E_1}\dots\ch_{k_r}^{E_r}}'
=\ta(\W\Phi_\bul)
=\ta_\W(\Phi_\bul).
\]

By \cref{ch as G}, $\Phi_\bul$ is a linear combination of products
\[
\fG_{\ell_1}(a_1)\dots\fG_{\ell_r}(a_r),
\qquad
0\le \ell_i\le k_i,
\]
where the $a_i$ are even cohomology classes.
By \cref{LQW:th}, every operator $\fG_k(a)$ is a linear combination
of operators \cref{p-expr}
\[
f_{j,r}(b)
=\sum_{\ov{\n\la=0}{\ell(\la)=r}}\frac{p_j(\la)}{\la!}\fa_\la(b),
\]
for $(j,r)=(0,k+2)$, $(0,k)$ or $(2,k)$ if $k\ge1$ and $(j,r)=(0,2)$ if $k=0$.
By \cref{pk:cor}, every such operator
is a linear combination of operators
\[T=[z^0]\nop{D^jA_{a_1}(z)\cdot A_{a_2}(z)\dots A_{a_r}(z)}\]
for suitable $a_1,\dots,a_r\in H$,
where
\[A_a(z)=\sum_{n\ne0}a_nz^{-n}=D\BF_a(z)\]
and $D^jA_a(z)=D^{j+1}\BF_a(z)$
are contained in the space $\D$ of bosonic descendants \cref{log des}.
The operator $T$ is admissible of weight $\le j+r$.
Therefore the operator $\fG_k(a)$ is admissible
of weight $\le k+2$.
By \cref{qm for adm oper:th} (see also \cref{rem:super qm}),
we obtain $\ta_\W(\fG_{\ell_1}(a_1)\dots\fG_{\ell_r}(a_r))\in\cC_0$
of weight at most $\sum_i(\ell_i+2)\le\sum_{i}(k_i+2)$.
Therefore $\ta_\W(\Phi_\bul)$ is quasi-modular of weight at most
$\sum_{i}(k_i+2)$.
\end{proof}

\subsection{Relation with zero modes of the Heisenberg vertex algebra}

A graded vertex algebra
(see \eg \cite{kac_vertex,frenkel_vertex,zhu_modular})
is a graded vector space $V=\bop_{n\ge0}V_n$
equipped with a linear map
\[Y:V\to\F^\gr(V),\qquad
a\mto Y(a,z)=\sum_{n\in\bZ}a_{(n)} z^{-n-1},
\]
and an element $\one\in V_0$ such that
\begin{enumerate}
\item $Y$ has degree zero, meaning that $\deg a_{(n)}=\deg a-n-1$ for homogeneous $a\in V$.
\item $Y(\one,z)=\id_V$.
\item $a_{(n)}\one=0$ for $n\ge0$ and $a_{(-1)}\one=a$.
\item $[T,Y(a,z)]=\dd_z Y(a,z)$,
where $T\in\End(V)$ is defined by $T(a)=a_{(-2)}\one$.
\item $(z-w)^n [Y(a,z),Y(b,w)]=0$ for $n\gg0$ (locality).
\end{enumerate}

For $a\in V_d$, the degree $\deg a=d$ is also called the conformal weight of $a$.
We have $\deg a_{(n)}=d-n-1$.
In particular, the operator
\eq{o(a):=a_{(d-1)}}
has degree zero.
Equivalently,
\[o(a)=[z^0]\wtl Y(a,z),\qquad \wtl Y(a,z)=Y(z^\bd a,z)\in\Fo(V),\]
where $\d:V\to V$ is the degree operator.
This defines a linear map
\[o:V\to\iend_0(V),\]
called the zero-mode map.
For homogeneous $a,b\in V$, the zero modes satisfy
\[
[o(a),o(b)]
=\sum_{j\ge0}\binom{\deg(a)-1}{j}o\rbr{a_{(j)}b}.
\]
Thus the image $o(V)\sbs\iend_0(V)$ is a Lie algebra.

\begin{remark}
The same expression represents the commutator in Zhu's algebra
$A(V)$ \cite{zhu_modular}
\[
[a]*[b]-[b]*[a]
=\sum_{j\ge0}\binom{\deg(a)-1}{j}\sbr {a_{(j)}b}.
\]
The zero-mode map on the space $V$, however, does not in general
factor through $A(V)$.
\end{remark}

Assuming that $\dim V_n<\infty$ for all $n\ge0$,
we define the $q$-trace
$\tr_q:\iend(V)\to K\lser q$
and the normalized $q$-trace $\ta:\iend(V)\to K\lser q$
in the same way as before.

\begin{remark}
If $V$ is moreover a vertex operator algebra with conformal vector
$\om$ and $Y(\om,z)=\sum_{n\in\bZ}L_nz^{-n-2}$,
then $\d=L_0$.
The central charge $c$ is defined by
\[
[L_m,L_n]=(m-n)L_{m+n}
+\frac{c}{12}(m^3-m)\de_{m,-n}\id_V.
\]
In \cite{zhu_modular}, for a graded module $M$ over $V$,
one considers the one-point trace functions
\[
Z_M(v,\rho)=\tr_M\bigl(q^{L_0-c/24}o(v)\bigr),
\qquad q=e^{2\pi i\rho},
\]
and proves modular transformation properties under suitable hypotheses.
For $M=V$,
\[Z_V(v,\rho)=q^{-c/24}\tr_V(q^\d o(v))=q^{-c/24}\tr_q(o(v)).\]
\end{remark}

In particular,
let $V=V(H)$ be the Fock space for a quadratic space $H$.
Then $V$ has a VOA structure, called the Heisenberg vertex algebra
(see \eg \cite{frenkel_vertex}).
Identifying $a\in H$ with $a_{-1}\one\in V$,
we have (note that $a_0=0\in\End(V)$)
\[Y(a,z)=\sum_{n\in\bZ}a_nz^{-n-1}.
\]
More generally, for $a^1,\dots,a^r\in H$ and $n_1,\dots,n_r\ge1$,
one has \cite[2.3.8]{frenkel_vertex}
\eq{\label{Y}
Y(a^1_{-n_1}\dots a^r_{-n_r}\one,z)
=\nop{\dd_z^{(n_1-1)}Y(a^1,z)\dots \dd_z^{(n_r-1)}Y(a^r,z)},
}
where $\dd_z^{(k)}=\dd_z^k/k!$.
The standard conformal vector is
$\om=\frac12\sum_i(e_i)_{-1}(e^i)_{-1}\one$,
where $(e_i)_i$ and $(e^i)_i$ are dual bases of $H$.
With this choice, the central charge is $c=\dim H$.

Using the Dedekind eta function $\eta(\rho)=q^{1/24}\ph q$, we obtain
\[
\ta(o(v))=\frac{\tr_q(o(v))}{\tr_q(1)}
=\ph q^c\tr_q(o(v))=\eta(\rho)^cZ_V(v,\rho).
\]

\begin{theorem}
The Lie algebra $o(V)$ coincides with the space of admissible operators
\[
\spn_K \sets{[z^0]\nop{X_1(z)\dots X_r(z)}}{r\ge0,\, X_i\in\D}.
\]
\end{theorem}
%


\begin{proof}
Define the polynomials
\[
P_n(x)=\prod_{i=1}^{n-1}(x-i),\qquad
B_n(x)
=\frac{P_n(x)}{(n-1)!}=\binom{x-1}{n-1}.
\]
For $D=z\dd_z$, we have
\[P_n(D)z^m=(m-1)\dots(m-n+1)z^m=z^n\dd_z^{n-1}z^{m-1},\]
hence $z^n\dd_z^{n-1}z\i=P_n(D)$.

For $a\in H$, we have $A_a(z)=D\BF_a(z)=\sum_{n\ne0}a_n z^{-n}$,
hence $Y(a,z)=z\i A_a(z)$.
Therefore
\[z^n\dd_z^{(n-1)}Y(a,z)=B_n(D) A_a(z).\]

The Fock space $V=V(H)$ is spanned by the states
\[
v=a^1_{-n_1}\cdots a^r_{-n_r}\one,
\qquad
a^i\in H,\, n_i\ge1,
\]
of degree (conformal weight) $d=n_1+\cdots+n_r$.
Applying \cref{Y}, we obtain
\[z^{d}Y(v,z)
=\nop{\prod*_{i=1}^r z^{n_i}\dd_z^{(n_i-1)}Y(a^i,z)}
=\nop{\prod*_{i=1}^rB_{n_i}(D)A_{a^i}(z)}.
\]
Therefore
\[
o(v)=[z^0]z^{d}Y(v,z)
=[z^0]\nop{\prod*_{i=1}^rB_{n_i}(D)A_{a^i}(z)}
\]
is an admissible operator.

Conversely, the polynomials $B_n(x)=\binom{x-1}{n-1}$ for $n\ge1$ form a basis of $K[x]$.
Let $c_{k,n}\in K$ be such that
$x^k=\sum_{n=1}^{k+1}c_{k,n}B_n(x)$ for $k\ge0$.
Then
\[
D^kA_a(z)=\sum*_{n=1}^{k+1}c_{k,n}B_n(D)A_a(z).
\]
Therefore every normally ordered monomial
\[
\nop{D^{k_1}A_{a^1}(z)\cdots D^{k_r}A_{a^r}(z)}
\]
is a finite linear combination of fields
\[\nop{\prod*_{i=1}^r B_{n_i}(D)A_{a^i}(z)}
=z^{\sum_i n_i}
\nop{\prod*_{i=1}^r \dd_z^{(n_i-1)}Y(a^i,z)}
=z^d Y(v,z)
\]
for $v=a^1_{-n_1}\cdots a^r_{-n_r}\one$ and $d=n_1+\dots+n_r$.
This implies that the admissible operator
\[[z^0]\nop{D^{k_1}A_{a^1}(z)\cdots D^{k_r}A_{a^r}(z)}
\]
is a linear combination of operators $o(v)$.
\end{proof}

The above result implies that if $v\in V_d$, then $o(v)$ is an admissible operator of weight at most $d$.

\begin{remark}
The same argument implies that the space of admissible fields
coincides with the image of the linear map $\wtl Y:V\to\Fo(V)$,
where $\wtl Y(v,z)=z^dY(v,z)$ for $v\in V_d$.
\end{remark}

\begin{corollary}
For all operators $f$ in the algebra $\ang{o(V)}_{alg}\sbs\iend_0(V)$ generated by $o(V)$, the normalized $q$-trace $\ta(f)$ is quasi-modular.
Moreover, if $v_i\in V_{d_i}$ for $i\in[r]$, then
$\ta(o(v_1)\dots o(v_r))$ is quasi-modular of weight at most $\sum_i d_i$.
\end{corollary}

\begin{remark}
Dong--Mason--Nagatomo \cite{dong_quasi}
prove for the Heisenberg vertex operator algebra that
\[
\ta(o(v))=\eta(\rho)^c Z_V(v,\rho)
\]
is quasi-modular.
Thus their theorem gives precisely the quasi-modularity of our
normalized $q$-trace for a single zero-mode insertion.
Our corollary gives the corresponding statement for arbitrary products
\[
\ta\rbr{o(v_1)\cdots o(v_r)}\in\QM(1,\bQ).
\]
This is a normalized zero-mode $r$-point correlator,
equal to the simultaneous constant term of the normalized $r$-point function
\[[z_1^0\dots z_r^0]
\frac{\tr\rbr{q^{L_0}\prod_iY(z_i^{L_0}v_i,z_i)}}{\tr\rbr{q^{L_0}}}.
\]
\end{remark}
\defcite\BO{bloch_character}
\subsection{Relation with Bloch--Okounkov quasi-modularity results}

Let $H=\bQ a$ with $(a,a)=1$, and let $V=V(H)$ be the corresponding
rank-one bosonic Fock space.  We identify
\[
V=S\left(\bigoplus_{n>0}\bQ a_{-n}\right)
\simeq \bQ[p_1,p_2,\dots]=\La,
\]
where $\La$ is the algebra of symmetric functions and
$a_{-n}=p_n$ are power sum symmetric functions.
For $n>0$, the  operator $a_{-n}$ acts by multiplication by $p_n$ and
$a_n=n\frac{\dd}{\dd p_n}$.

The vector space $\La$ has a basis
consisting of Schur functions
\[\ket\la=s_\la\in\La,\qquad \la\in\Par,
\]
corresponding to the standard charge-zero infinite-wedge basis under the boson--fermion correspondence (see \eg
\cite[\S5.2]{kac_vertex} or \cite[\S5.3]{frenkel_vertex}).
The element $\ket\la$ has degree (conformal weight) $\n\la=\sum_{i\ge1}\la_i$.
Given a function $f:\Par\to\bQ$,
define the operator
\eq{\hat f\in\iend_0(V),\qquad
\hat f\ket\la=f(\la)\ket\la.}
Then the $q$-bracket \BO[(0.16)]
\eq{\ang f_q=\sum_\la f(\la)q^{\n\la}/\sum_\la q^{\n\la}
=\ta\rbig{\hat f}}
is precisely the normalized $q$-trace of $\hat f$.
In \BO one proves quasi-modularity of $q$-brackets of a certain class of functions $f:\Par\to\bQ$.
We are going to relate this statement to our quasi-modularity result for normalized $q$-traces.

The Bloch--Okounkov operators $\tD_k$, $k\ge1$,
act on $\La$ by
\eq{
\tD_k\ket\la =\rbr{p_k(\la)-\ksi(-k)} \ket\la,
}
where \BO[(0.11)]
\[
\ksi(s)=(2^s-1)\ze(s)
\]
and $p_k(\la)$ are the shifted power sums \cite[(0.13)]{bloch_character},
which can be characterized as follows.
Let
\eq{ \label{W la}
W_\la(s)
=\sum_{i\ge1}e^{s(\la_i-i+\frac12)}
=\sum_{k\ge0}Q_k(\la)s^{k-1},
}
in particular, $Q_0(\la)=1$ and $Q_1(\la)=0$.
Then \cite[(0.18)]{bloch_character}
\eq{\label{Q}
p_k(\la)-\ksi(-k)
=\left. \dd_s^k\rbr{W_\la(s)-\frac1s} \right|_{s=0}
=k!Q_{k+1}(\la).
}
Consequently the generating operator
\[
\cT(s)
=
\frac1s\id+\sum_{k\ge1}\tD_k\frac{s^k}{k!}
\]
is characterized by
\[
\cT(s)\ket\la=W_\la(s)\ket\la.
\]

\begin{proposition}\label{BO bosonization}
Under the boson--fermion correspondence, we have
\[
\cT(s)=\frac1{\vsi(s)} [z^0]\nop{e^{X_s(z)}},
\]
where
\[
\vsi(s)=e^{s/2}-e^{-s/2},\qquad
X_s(z)
=\sum_{n\ne0}\frac{\vsi(ns)}n a_nz^{-n}.
\]
\end{proposition}

\begin{proof}
Let
\[
\psi(z)=\sum_{r\in\bZ+\frac12}\psi_rz^r,
\qquad
\psi^*(z)=\sum_{r\in\bZ+\frac12}\psi_r^*z^{-r}
\]
be the fermionic fields.
The standard boson--fermion correspondence gives
\[
\psi(e^{s/2}z)\psi^*(e^{-s/2}z)
=
\frac1{\vsi(s)}
\nop{\exp\left(
\sum_{n\ne0}\frac{\vsi(ns)}n a_nz^{-n}
\right)}.
\]
Taking the constant coefficient in $z$ gives
\[
\frac1{\vsi(s)}
[z^0]\nop{e^{X_s(z)}}
=
\sum_{r\in\bZ+\frac12}e^{sr}\psi_r\psi_r^*.
\]
The occupied positions of $\ket\la$ are
$r_i=\la_i-i+\frac12$ for $i\ge1$.
Hence
\[
\sum_{r\in\bZ+\frac12}e^{sr}\psi_r\psi_r^*\ket\la
=\rbr{\sum_{i\ge1}e^{s(\la_i-i+\frac12)}} \ket\la
=W_\la(s)\ket\la.
\]
The result follows from the characterization of $\cT(s)$ above.
\end{proof}

The preceding formula expresses the Bloch--Okounkov operators $\tD_k$ directly in terms of our bosonic descendants.
Indeed, consider $D=z\dd_z$ and the bosonic descendant
\[
A(z)=D\BF_a(z)=\sum_{n\ne0}a_nz^{-n}.
\]
Then
\[
X_s(z)
=\sum_{m\ge0} \frac{s^{2m+1}}{2^{2m}(2m+1)!}\,
D^{2m}A(z).
\]
Thus $X_s(z)$ is a formal linear combination of elements of $\D$,
and the coefficient of every power of $s$ in
$[z^0]\nop{e^{X_s(z)}}$
is a finite linear combination of admissible operators.
Since
\[
\frac1{\vsi(s)}=\frac1s+\sum_{j\ge1}c_js^{2j-1},
\qquad c_j\in\bQ,
\]
we obtain the following.

\begin{corollary}\label{BO operators admissible}
For every $k\ge1$, the Bloch--Okounkov operator $\tD_k$ is admissible,
of weight at most $k+1$.
\end{corollary}

\begin{proof}
The coefficient of $s^{2m+1}$ in $X_s(z)$ is a scalar multiple of
$D^{2m}A(z)$, which has weight $2m+1$.
Hence the coefficient of $s^N$ in
$\nop{e^{X_s(z)}}$ is a linear combination of normally ordered
monomials of weight $N$.
Multiplication by $\vsi(s)^{-1}$ shows that the coefficient of $s^k$
in $\cT(s)$ has weight at most $k+1$.
\end{proof}

By \cref{qm for adm oper:th},
we conclude that for $k_1,\dots,k_r\ge1$,
the $q$-trace
\eq{
\ta\rbr{\tD_{k_1}\cdots\tD_{k_r}}
=\ang{
\prod*_{i=1}^r \rbr {p_{k_i}-\ksi(-k_i)} }_q
}
is quasi-modular of weight at most $\sum_i(k_i+1)$.
Using the functions $Q_k:\Par\to\bQ$,
which satisfy $k!Q_{k+1}(\la)=p_k(\la)-\ksi(-k)$ \cref{Q},
we obtain that
\eq{
\ang{Q_{k_1}\dots Q_{k_r}}_q,\qquad k_i\ge2,
}
is quasi-modular of weight at most $\sum_i k_i$.

Consider the algebra of shifted symmetric functions
\[
\La^*=\bQ[Q_2,Q_3,\dots],
\qquad
\wt(Q_k)=k.
\]
We conclude that for every homogeneous $f\in\La^*$ of weight $k$,
the $q$-bracket $\ang f_q$ is quasi-modular of weight at most $k$.
Thus the Bloch--Okounkov quasi-modularity theorem for shifted
symmetric functions follows from our quasi-modularity theorem for
admissible operators.

\bibliographystyle{halpha}
\bibliography{q-trace}

@Article{zhu_modular,
  author     = {Zhu, Yongchang},
  journal    = {J. Amer. Math. Soc.},
  title      = {{M}odular invariance of characters of vertex operator algebras},
  year       = {1996},
  doi        = {10.1090/S0894-0347-96-00182-8},
  issn       = {0894-0347},
  number     = {1},
  pages      = {237--302},
  url        = {http://dx.doi.org/10.1090/S0894-0347-96-00182-8},
  volume     = {9},
  eprinttype = {arxiv},
  fjournal   = {Journal of the American Mathematical Society},
  mrclass    = {17B69 (11F22 17B68)},
  mrnumber   = {1317233 (96c:17042)},
  mrreviewer = {Chong Ying Dong},
}

@Book{lepowsky_introduction,
  author     = {Lepowsky, James and Li, Haisheng},
  title      = {Introduction to vertex operator algebras and their representations},
  doi        = {10.1007/978-0-8176-8186-9},
  eprinttype = {arxiv},
  isbn       = {0-8176-3408-8},
  pages      = {xiv+318},
  publisher  = {Birkh\"{a}user Boston, Inc., Boston, MA},
  series     = {Progress in Mathematics},
  url        = {https://doi.org/10.1007/978-0-8176-8186-9},
  volume     = {227},
  mrclass    = {17B69 (17B67 17B68 81R10)},
  mrnumber   = {2023933},
  mrreviewer = {Dra\v{z}en Adamovi\'{c}},
  year       = {2004},
}

@Article{alhwaimel_hilbert,
  author     = {Alhwaimel, Mazen M. and Qin, Zhenbo},
  journal    = {Pure Appl. Math. Q.},
  title      = {Hilbert schemes of points on surfaces and multiple {$q$}-zeta values},
  year       = {2024},
  doi        = {10.4310/pamq.241112043526},
  eprint     = {2310.10812},
  issn       = {1558-8599,1558-8602},
  number     = {6},
  pages      = {2615--2646},
  url        = {https://doi.org/10.4310/pamq.241112043526},
  volume     = {20},
  MRCLASS    = {14C05 (11F23 11M32)},
  MRREVIEWER = {Scott\ R.\ Nollet},
  eprinttype = {arxiv},
  fjournal   = {Pure and Applied Mathematics Quarterly},
  mrnumber   = {4829861},
}

@Article{bloch_character,
  author     = {Bloch, Spencer and Okounkov, Andrei},
  journal    = {Adv. Math.},
  title      = {The character of the infinite wedge representation},
  year       = {2000},
  doi        = {10.1006/aima.1999.1845},
  eprint     = {alg-geom/9712009},
  issn       = {0001-8708,1090-2082},
  number     = {1},
  pages      = {1--60},
  url        = {https://doi.org/10.1006/aima.1999.1845},
  volume     = {149},
  MRCLASS    = {11F22 (17B66)},
  MRREVIEWER = {Mirko\ Primc},
  fjournal   = {Advances in Mathematics},
  mrnumber   = {1742353},
}

@Article{okounkov_hilbert,
  author     = {Okounkov, Andrei},
  journal    = {Funktsional. Anal. i Prilozhen.},
  title      = {Hilbert schemes and multiple {$q$}-zeta values},
  year       = {2014},
  doi        = {10.1007/s10688-014-0054-z},
  eprint     = {1404.3873},
  issn       = {0374-1990},
  number     = {2},
  pages      = {79--87},
  url        = {https://doi.org/10.1007/s10688-014-0054-z},
  volume     = {48},
  eprinttype = {arxiv},
  fjournal   = {Funktsional\cprime ny\u{\i} Analiz i ego Prilozheniya},
  mrclass    = {14C05 (11B65 14N10)},
  mrnumber   = {3288178},
  mrreviewer = {\c{S}erban B\u{a}rc\u{a}nescu},
}

@Article{carlsson_vertex,
  author        = {Carlsson, Erik},
  title         = {Vertex Operators and Moduli Spaces of Sheaves},
  year          = {2009},
  eprint        = {0906.1825},
  note          = {PhD thesis},
  archiveprefix = {arXiv},
}

@Article{alhwaimel_qins,
  author        = {Alhwaimel, Mazen M.},
  journal       = {Manuscripta Math.},
  title         = {Toward {Q}in's conjecture on {H}ilbert schemes of points and quasi-modular forms},
  year          = {2025},
  doi           = {10.1007/s00229-025-01630-1},
  eprint        = {2407.01554},
  issn          = {0025-2611,1432-1785},
  number        = {3},
  pages         = {Paper No. 31, 40},
  url           = {https://doi.org/10.1007/s00229-025-01630-1},
  volume        = {176},
  MRCLASS       = {14C05 (11M32)},
  MRREVIEWER    = {Roberto\ Notari},
  archiveprefix = {arXiv},
  fjournal      = {Manuscripta Mathematica},
  mrnumber      = {4900530},
}

@Misc{qin_quick,
  author = {Qin},
  title  = {A quick survey from {S}-duality conjecture of {V}afa-{W}itten to a conjecture of {O}kounkov},
  note   = {Presentation at the SQuaRE Workshop “Moduli of sheaves on surfaces via Bridgeland stability”, American Institute of Mathematics, San Jose, CA},
  year   = {2022},
}

@Article{carlsson_vertexa,
  author   = {Carlsson, Erik},
  journal  = {Adv. Math.},
  title    = {Vertex operators and quasimodularity of {C}hern numbers on the {H}ilbert scheme},
  year     = {2012},
  doi      = {10.1016/j.aim.2011.10.003},
  issn     = {0001-8708},
  number   = {5},
  pages    = {2888--2907},
  url      = {https://doi.org/10.1016/j.aim.2011.10.003},
  volume   = {229},
  fjournal = {Advances in Mathematics},
  mrclass  = {14C05 (17B69)},
  mrnumber = {2889150},
}

@Book{deligne_quantum,
  title      = {{Q}uantum fields and strings: a course for mathematicians. {V}ol. 1},
  editor     = {Deligne, Pierre and Etingof, Pavel and Freed, Daniel S. and Jeffrey, Lisa C. and Kazhdan, David and Morgan, John W. and Morrison, David R. and Witten, Edward},
  eprinttype = {arxiv},
  isbn       = {0-8218-1198-3},
  note       = {Material from the Special Year on Quantum Field Theory held at the Institute for Advanced Study, Princeton, NJ, 1996--1997},
  pages      = {Vol. 1: xxii+723 pp.; Vol. 2: pp. i--xxiv and 727--1501},
  publisher  = {American Mathematical Society},
  address    = {Providence, RI},
  mrclass    = {81-06 (81T30 81Txx)},
  mrnumber   = {1701618 (2000e:81010)},
  year       = {1999},
}

@Article{dong_quasi,
  author   = {Dong, Chongying and Mason, Geoffrey and Nagatomo, Kiyokazu},
  journal  = {International Mathematics Research Notices},
  title    = {Quasi-modular forms and trace functions associated to free boson and lattice vertex operator algebras},
  year     = {2001},
  doi      = {10.1155/S1073792801000204},
  issn     = {1073-7928,1687-0247},
  number   = {8},
  pages    = {409--427},
  mrnumber = {1827085},
}

@Book{kac_vertex,
  author    = {Kac, Victor},
  title     = {Vertex algebras for beginners},
  doi       = {10.1090/ulect/010},
  edition   = {Second},
  isbn      = {0-8218-1396-X},
  pages     = {vi+201},
  publisher = {American Mathematical Society, Providence, RI},
  series    = {University Lecture Series},
  url       = {https://doi.org/10.1090/ulect/010},
  volume    = {10},
  MRCLASS   = {17B69},
  mrnumber  = {1651389},
  year      = {1998},
}

@Book{frenkel_vertex,
  author     = {Frenkel, Edward and Ben-Zvi, David},
  title      = {Vertex algebras and algebraic curves},
  doi        = {10.1090/surv/088},
  edition    = {Second},
  eprinttype = {arxiv},
  isbn       = {0-8218-3674-9},
  pages      = {xiv+400},
  publisher  = {American Mathematical Society, Providence, RI},
  series     = {Mathematical Surveys and Monographs},
  url        = {https://doi.org/10.1090/surv/088},
  volume     = {88},
  mrclass    = {17B69 (14D21 14H81)},
  mrnumber   = {2082709},
  year       = {2004},
}

@Article{lehn_cup,
  author     = {Lehn, Manfred and Sorger, Christoph},
  journal    = {Invent. Math.},
  title      = {The cup product of {H}ilbert schemes for {$K3$} surfaces},
  year       = {2003},
  doi        = {10.1007/s00222-002-0270-7},
  eprint     = {math/0012166},
  issn       = {0020-9910,1432-1297},
  number     = {2},
  pages      = {305--329},
  url        = {https://doi.org/10.1007/s00222-002-0270-7},
  volume     = {152},
  eprinttype = {arxiv},
  fjournal   = {Inventiones Mathematicae},
  mrclass    = {14C05 (14J28)},
  mrnumber   = {1974889},
  mrreviewer = {Roy\ M.\ Skjelnes},
}

@Article{carlsson_exts,
  author        = {Carlsson, Erik and Okounkov, Andrei},
  journal       = {Duke Math. J. 161, no. 9 (2012), 1797-1815},
  title         = {Exts and Vertex Operators},
  year          = {2012},
  date          = {2008-01-16},
  doi           = {10.1215/00127094-1593380},
  eprint        = {0801.2565},
  issn          = {0012-7094},
  number        = {9},
  volume        = {161},
  archiveprefix = {arXiv},
  copyright     = {arXiv.org perpetual, non-exclusive license},
  month         = jun,
  publisher     = {Duke University Press},
}

@Article{li_hilbert,
  author     = {Li, Wei-Ping and Qin, Zhenbo and Wang, Weiqiang},
  journal    = {Intern. Math. Res. Notices 27 (2002) 1427--1456.},
  title      = {Hilbert schemes and {W} algebras},
  year       = {2001},
  eprint     = {math/0111047},
  eprinttype = {arxiv},
}

@Article{mason_torus,
  author        = {Mason, Geoffrey and Tuite, Michael P.},
  journal       = {Commun.Math.Phys. 235 (2003) 47-68},
  title         = {Torus Chiral n-Point Functions for Free Boson and Lattice Vertex Operator Algebras},
  year          = {2002},
  date          = {2002-04-26},
  doi           = {10.1007/s00220-002-0772-6},
  eprint        = {math/0204323},
  issn          = {1432-0916},
  number        = {1},
  pages         = {47--68},
  volume        = {235},
  archiveprefix = {arXiv},
  copyright     = {Assumed arXiv.org perpetual, non-exclusive license to distribute this article for submissions made before January 2004},
  month         = Apr,
  primaryclass  = {math.QA},
  publisher     = {Springer Science and Business Media LLC},
}

@Article{alhwaimel_equivariant,
  author        = {Alhwaimel, Mazen M. and Qin, Zhenbo},
  journal       = {Pure Appl. Math. Q.},
  title         = {Equivariant {C}hern character operators and {O}kounkov's conjecture},
  year          = {2026},
  doi           = {10.4310/pamq.260610205021},
  eprint        = {2505.14626},
  issn          = {1558-8599,1558-8602},
  number        = {3},
  pages         = {1075--1113},
  url           = {https://doi.org/10.4310/pamq.260610205021},
  volume        = {22},
  MRCLASS       = {14C05 (05E10 11B65 11M32 17B69)},
  archiveprefix = {arXiv},
  fjournal      = {Pure and Applied Mathematics Quarterly},
  mrnumber      = {5106144},
}

@Misc{alekseev_qins,
  author        = {Alekseev, Victor and Chakravarty, Avik and Choi, Daebeom and Xu, Shengjing},
  title         = {Qin's quasimodularity conjecture for {H}ilbert schemes of points},
  year          = {2026},
  archiveprefix = {arXiv},
  eprint        = {2609.33884},
}

@Article{goujard_counting,
  author     = {Goujard, Elise and M\"{o}ller, Martin},
  journal    = {J. Eur. Math. Soc. (JEMS)},
  title      = {Counting {F}eynman-like graphs: quasimodularity and {S}iegel-{V}eech weight},
  year       = {2020},
  issn       = {1435-9855,1435-9863},
  number     = {2},
  pages      = {365--412},
  volume     = {22},
  MRCLASS    = {05E10 (11F11 14J33 14T15 32G15)},
  MRREVIEWER = {Ayberk\ Zeytin},
  doi        = {10.4171/jems/924},
  fjournal   = {Journal of the European Mathematical Society (JEMS)},
  mrnumber   = {4049220},
  url        = {https://doi.org/10.4171/jems/924},
}

\end{document}